\documentclass[11pt,dvips,letterpaper]{article}

\usepackage{pslatex}
\usepackage{fancyhdr}
\usepackage{graphicx}
\usepackage{geometry}

\RequirePackage{amsfonts,amssymb,amsmath,amscd,amsthm}
\RequirePackage{txfonts}
\RequirePackage{graphicx}
\RequirePackage{xcolor}
\RequirePackage{enumerate}

\def\figurename{Figure} 
\makeatletter
\renewcommand{\fnum@figure}[1]{\figurename~\thefigure.}
\makeatother

\def\tablename{Table} 
\makeatletter
\renewcommand{\fnum@table}[1]{\tablename~\thetable.}
\makeatother

\usepackage{amsmath}
\usepackage{amssymb}
\usepackage{amsfonts}
\usepackage{amsthm,amscd}

\newtheorem{theorem}{Theorem}[section]
\newtheorem{lemma}[theorem]{Lemma}
\newtheorem{corollary}[theorem]{Corollary}
\newtheorem{proposition}[theorem]{Proposition}
\theoremstyle{definition}
\newtheorem{definition}[theorem]{Definition}
\newtheorem{example}[theorem]{Example}

\theoremstyle{remark}
\newtheorem{remark}[theorem]{Remark}

\numberwithin{equation}{section}

\def\C{\mathbb C}

\def\R{\mathbb R}

\def\K{\mathbb K}

\def\N{\mathbb N}

\newcommand{\D}{\mathcal D}

\newcommand{\G}{\mathcal{G}}

\newcommand{\der}{\text{der}}
\newcommand{\sspan}{\text{span}}

\newcommand{\s}{\mathfrak s}

\newcommand{\beqn}{\begin{eqnarray}}
\newcommand{\eeqn}{\end{eqnarray}}
\newcommand{\beq}{\begin{equation}}
\newcommand{\eeq}{\end{equation}}
\newcommand{\bpro}{\begin{proposition}}
\newcommand{\epro}{\end{proposition}}
\newcommand{\blem}{\begin{lemma}}
\newcommand{\elem}{\end{lemma}}
\newcommand{\bdfn}{\begin{definition}}
\newcommand{\edfn}{\end{definition}}
\newcommand{\bcor}{\begin{corollary}}
\newcommand{\ecor}{\end{corollary}}
\newcommand{\bthm}{\begin{theorem}}
\newcommand{\ethm}{\end{theorem}}
\newcommand{\bex}{\begin{example}}
\newcommand{\eex}{\end{example}}
\newcommand{\brmq}{\begin{remark}}
\newcommand{\ermq}{\end{remark}}
\newcommand{\benum}{\begin{enumerate}}
\newcommand{\eenum}{\end{enumerate}}
\newcommand{\bitem}{\begin{itemize}}
\newcommand{\eitem}{\end{itemize}}

\begin{document}
\title{\bfseries\scshape{Automorphisms of cotangent bundles of Lie groups}}
\author{\bfseries\scshape  Andr\'e DIATTA \thanks{E-mail address: \tt{andre.diatta@fresnel.fr, andrediatta@gmail.com}}\\
 CNRS, Institut Fresnel, UMR 7249, \\
Aix-Marseille Universit\'e, Centrale Marseille,
\\ 13013 Marseille, France.
\\\bfseries\scshape Bakary MANGA\thanks{E-mail address: \tt{ bakary.manga@imsp-uac.org, bakary.manga@ucad.edu.sn}}\\
D\'epartement de Math\'ematiques et Informatique, \\ 
 Universit\'e Cheikh Anta Diop de Dakar,\\
BP 5005 Dakar-Fann, Dakar, S\'en\'egal
\\ and\\
Institut de Math\'ematiques et de Sciences Physiques (IMSP)
\\
 01 BP 613, Porto-Novo,
B\'enin.}

\date{}
\maketitle \thispagestyle{empty} \setcounter{page}{1}

\thispagestyle{fancy} \fancyhead{}
\fancyfoot{}
\renewcommand{\headrulewidth}{0pt}

\begin{abstract}
Let $G$ be a Lie group, $\G$ its Lie algebra and $T^*G$ its cotangent bundle. 
On  $T^*G,$ we consider the Lie group structure obtained by performing a left 
trivialization and endowing the resulting trivial bundle  $G\times\G^*$ with the 
semi-direct product, using the co-adjoint action of $G$ on the dual space  $\G^*$ of  $\G$.
 We investigate the group of automorphisms of the Lie algebra $\D:=T^*\G$ of $T^*G.$ More precisely, 
 we fully characterize the Lie algebra of all derivations of $\D,$ exhibiting a finer decomposition into 
 components made of well known spaces. Further, we specialize to the cases where $G$ has a  bi-invariant Riemannian 
 or pseudo-Riemannian metric, with the semi-simple and compact cases investigated as particular cases.
\end{abstract}

\noindent {\bf AMS Subject Classification:} 22C05; 22E60; 22E15; 22E10.

\vspace{.08in} \noindent \textbf{Keywords}: Cotangent bundle, automorphism, derivation, Lie group,
Lie algebra, bi-invariant metric, bi-invariant tensor, supersymmetric Lie group, Lie superalgebra, Lie supergroup.

\section{Introduction}

Throughout this work, all Lie groups and Lie algebras are over $\mathbb K=\mathbb R$ or $\mathbb C,$ unless otherwise specified. 
Let G be a Lie group whose Lie algebra $\G $ is identified with its tangent space $T_{\epsilon}G$ at the unit
$\epsilon$. We always look at the cotangent bundle $T^*G$  of $G$, as a Lie group obtained by performing the semi-direct product $G\ltimes\G^*$
of $G$ and the Abelian Lie group $\mathcal G^*$, where $G$ acts on the
dual space $\G^*$ of $\G$ via the co-adjoint action. Here, using the
trivialization by left translations, the manifold underlying $T^*G$ has been
identified with the trivial  bundle $G\times \G^*$. We sometimes refer to the above Lie group structure on $T^*G,$ as 
its natural Lie group structure.  The Lie algebra  $Lie(T^*G)=\G\ltimes
 \G^*$ of $T^*G$ will be denoted by $T^*\G$ or simply by $\D$.

It is our aim in this work to study the connected component of the unit
of the group Aut($\mathcal D$) of automorphisms of the Lie algebra $\mathcal D$.
Such a connected component is precisely the connected subgroup of Aut($\mathcal D$), spanned by the exponentials of the derivations of
$\D$. Therefore, we will work locally, that is, with those derivations and the first
cohomology space $H^1(\D,\D)$, where $\D$ is seen as a $\D$-module for the adjoint representation.

Our motivation for this work comes from several interesting mechanical, algebraic and geometric problems.
Cotangent bundles of Lie groups appear amongst key tools in various mathematical formulations for the motion of particles 
(see e.g. \cite{alekseevsky-grabowski94}, \cite{marsden-ratiu-weinstein}, \cite{montgomery}). In particular they are the configuration 
space of some mechanical systems (see e.g. \cite{marle}). They can also exhibit rich algebraic and 
geometric structures (see e.g. \cite{bajo-benayadi-medina},  \cite{di-me-cybe}, \cite{drinfeld}, \cite{feix}, \cite{kronheimer}, \cite{marle}).
All such structures can be better understood when one can compare, deform or
classify them. This very often involves the invertible homomorphisms
(automorphisms) of $T^*G$, if in particular, such structures are invariant under left or right multiplications by the  
elements of $T^*G$. The derivatives at the unit of automorphisms of the Lie group $T^*G$ are
automorphisms of the Lie algebra $\D.$ Conversely, if $G$ is connected and simply connected,
then so is $T^*G$ and every automorphism of the Lie algebra $\D$
integrates to an automorphism of the Lie group $T^*G$. A problem involving left or right  invariant
structures on a Lie group also usually transfers to one on its Lie algebra, with the Lie
algebra automorphisms used as a means to compare or classify the corresponding induced structures.

In the purely algebraic point of view, finding and understanding the
derivations of a given Lie algebra or superalgebra, is in itself an important
problem  (\cite{burde2002}, \cite{burde2012}, \cite{deruiter}, \cite{dixmier57}, \cite{feldvoss}, \cite{jacobson55}, \cite{jacobson37},
\cite{leger63}, \cite{peng2014}, \cite{togo67}).

On the other hand, as a Lie group, the cotangent bundle $T^*G$ is a common
Drinfel'd double Lie group for all exact Poisson-Lie structures given by
solutions of the Classical Yang-Baxter Equation in $G$. See e.g. \cite{di-me-poisson}.
Double Lie algebras (resp. groups) encode information on
integrable Hamiltonian systems and Lax pairs (\cite{babelon-viallet}, \cite{bordemann}, \cite{drinfeld}, \cite{lu-weinstein}), 
Poisson homogeneous spaces of Poisson-Lie groups and the symplectic foliation of the corresponding Poisson structures 
(\cite{di-me-poisson}, \cite{drinfeld}, \cite{lu-weinstein}). To that extend, the description of the group of 
automorphisms of the double 
Lie algebra of a Poisson-Lie structure might be a contribution towards solving some of the open  problems in that field.

We find that the space of derivations of  $\D$
encompasses  spaces of  known and wildly used operators, amongst
which the derivations of  $\G$, the second space of the left invariant de
Rham cohomology $H^2_{inv}(G,\mathbb K)$ of $G$ (see Section \ref{chap:cocyles}), 
bi-invariant endomorphisms, in particular operators giving rise to complex group structure in
$G$, when they exist.



Traditionally, spectral sequences are used as a powerful tool for the study  of the cohomology spaces of extensions of Lie groups 
or Lie algebras and more generally, of locally trivial fiber bundles (see e.g. \cite{hochschild-serre:coho-algebras}, 
\cite{hochschild-serre:coho-ext-groups}, \cite{knapp}, \cite{neeb} for wider discussions). However, for the purpose of this investigation, 
we use a direct approach. 

The paper is organized as follows. 
In Section \ref{chap:preliminaries}, we explain some of the material and terminology needed to make the paper more self contained. 
Sections \ref{chap:automorphisms} and \ref{chap:orthogonal-algebra} are the actual core of the work where the main calculations 
and proofs of theorems are carried out. We discuss some examples in low dimension in Section \ref{section:examples}. We will see that some of those
examples give rise to contact Lie algebras. See \cite{diatta-contact}, \cite{diatta-contact-riemann}, for wider discussions on contact Lie algebras.

\pagestyle{fancy} \fancyhead{} \fancyhead[EC]{A. Diatta and B. Manga} \fancyhead[EL,OR]{\thepage} \fancyhead[OC]{Automorphisms of cotangent bundles of Lie groups} \fancyfoot{}
\renewcommand\headrulewidth{0.5pt}

\section{Preliminaries}\label{chap:preliminaries}

Although not central to the purpose of the work within this paper, the following material might be useful, 
at least, as regards parts of the terminology used throughout this paper.


\medskip

Throughout this paper, given a Lie group $G$, we will always let $G\ltimes \G^*$ stand for the Lie group consisting of the 
Cartesian product $G\times \G^*$ as its underlying manifold, together with the group structure obtained by semi-direct product using the 
co-adjoint action of $G$ on $\G^*.$
Recall that  the trivialization by left translations, or simply the left trivialization of $T^*G$ is given by the following 
isomorphism $\zeta$ of vector bundles,
$\zeta: T^*G \to G\times \G^*$, $(\sigma,\nu_\sigma)\mapsto (\sigma,\nu_\sigma\circ T_\epsilon L_\sigma),$ 
 where $L_\sigma$ is the left multiplication  $L_\sigma:G\to G$,  $\tau\mapsto L_\sigma(\tau):=\sigma\tau$ by $\sigma$ in $G$ and $T_\epsilon L_\sigma$ is the derivative of $L_\sigma$ at the unit $\epsilon.$
In this paper, $T^*G$ will always be endowed with the Lie group structure such that  $\zeta$ is an isomorphism of Lie groups. 
The Lie algebra of $T^*G$ is then the semi-direct product $\D:=\G\ltimes\G^*.$ More precisely, the Lie bracket on $\D$ reads: 
for any two elements $(x,f)$ and $(y,g)$ of $\D$,
\beq \label{bracket_double}
[(x,f),(y,g)]:=([x,y],ad^*_xg-ad^*_yf).
\eeq


\medskip


We will  refer to an object which is invariant
under both left and right translations in a Lie group $G$, as a
bi-invariant object. We discuss in this section, how $T^*G$ is naturally endowed with a  bi-invariant  pseudo-Riemannian metric.

 A bi-invariant (Riemannian or pseudo-Riemannian) metric in a connected
Lie group $G$, corresponds to a symmetric bilinear non-degenerate scalar form $\mu$
in its Lie algebra $\G$, such that the adjoint representation of $\G$ lies in the Lie algebra
$\mathcal O(\G, \mu)$ of infinitesimal isometries of $\mu$, or equivalently
\beq\label{ad-invariant}
\mu([x, y], z) + \mu(y, [x, z]) = 0,
\eeq
for any $ x, y, z$ in $\G$. This means that $\mu$ is invariant under the adjoint action of $\G$.
In this case, we will simply say that $\mu$ is adjoint-invariant.
Lie groups with bi-invariant metrics and their Lie algebras are called
{\em orthogonal or quadratic} (see e.g. \cite{me-re85}, \cite{me-re93}). Every orthogonal
Lie group is obtained by a series of so-called {\em double extension}
constructions introduced by Medina and Revoy (\cite{me-re85}, \cite{me-re93}).
Consider the isomorphism of vector spaces $\theta : \G \to \G^*$ defined
by $\langle \theta(x),y \rangle := \mu(x,y)$, where $\langle, \rangle$ on
the left hand side is the duality pairing $\langle x,f \rangle=f(x)$ between elements $x$ of $\G$ and $f$ of $\G^*$. 
Then, $\theta$ is an isomorphism of $\G$-modules in
the sense that it is equivariant with respect to the adjoint and co-adjoint actions of
$\G$ on $\G$ and $\G^*$ respectively; {\it i.e.} for any $x$ in $\G$,
$
\theta \circ  ad_x =  ad^*_x \circ \theta.
$
We also have, for all element $x$ of $\G$, $\theta^{-1} \circ ad^*_x = ad_x \circ \theta^{-1}$.
 The converse is also true. Precisely, a Lie group (resp. algebra) is
orthogonal if and only if its adjoint and co-adjoint representations are isomorphic.
See Theorem 1.4. of \cite{me-re93}.

Semisimple Lie groups (with their Killing form), compact Lie groups, or more generally, reductive Lie groups, 
are examples of orthogonal Lie groups.

The cotangent bundle of any Lie group (with its natural Lie group structure, as above) and in
general any element of the larger family of the so-called Drinfel'd doubles,  are orthogonal Lie groups  (\cite{drinfeld}),
as explained below.  As above, let  $\mathcal D:=\G \ltimes \G^*$ be the Lie algebra of $T^*G$ with the Lie bracket given 
in (\ref{bracket_double}). Let $\mu_0$ stand for the duality pairing $\langle, \rangle$, that is, for all  $(x,f), \; (y,g)$ in $\D$,
\beq\label{eq:dualitypairing}
\mu_0\Big((x,f),(y,g)\Big) = f(y) + g(x).
\eeq
 Then, $\mu_0$ satisfies the property (\ref{ad-invariant})  on $\D$ and hence gives rise to a
 bi-invariant (pseudo-Riemannian) metric on $T^*G$.

\section{Group of automorphisms of $\D:=T^*\G$}\label{chap:automorphisms}

\subsection{Derivations of $\D:=T^*\G$}\label{subsect:derivations-cotangent}

Consider a Lie group $G$ of dimension $n$, with Lie algebra $\G,$ over $\mathbb K$.
Let us also denote by $\D$ the vector space underlying the Lie algebra
$\mathcal D$ of the cotangent bundle $T^*G$, regarded as a $\D$-module under the adjoint action of $\D$. Consider the following
complex with the coboundary operator $\partial$,
$0\!\! \to\!\! \D\!\! \to\!\! Hom(\mathcal D,\D)\!\! \to\!\! Hom(\Lambda^2\mathcal D,\D)\!\!  \to\!\! \cdots
\!\!\to\!\! Hom(\Lambda^{2n}\mathcal D,\D)\!\! \to\!\! 0.
$ 
The coboundary $\partial \phi$ of the element $\phi$ of $Hom(\mathcal D,\D)\!\!:=\!\! \{\phi\! :\!\! \mathcal D\!\! \to\!\! \D, 
\phi \mbox{ linear} \}$ is the element of $Hom(\Lambda^2\mathcal D,\D)$ defined by
$
\partial \phi (u,v)\!:=\! ad_u\big(\phi(v)\big)\! -\! ad_v\big(\phi(u)\big)\! - \!\phi([u,v]),
$
for any elements $u\!=\!(x,f)$ and $v\!=\!(y,g)$ in $\mathcal D$. An element $\phi$ of $Hom(\mathcal D,\D)$
is a $1$-cocycle if $\partial \phi = 0$, {\it i.e. }
\beqn\label{suz1}
\phi([u,v]) = ad_u\big(\phi(v)\big) - ad_v\big(\phi(u)\big)
            = [u,\phi(v)]+[\phi(u),v],
\eeqn
for all elements $u,v$ of $\mathcal D$. In other words,  $1$-cocycles are the derivations of the Lie
algebra $\mathcal D$. In Section \ref{chap:cohomology}, we will characterize the first cohomology space 
$H^1(\D,\D):=\ker(\partial^2)/Im(\partial^1)$ of the associated Chevalley-Eilenberg cohomology, where for clarity, we have 
denoted by $\partial^1$ and $\partial^2$ the  restrictions of the coboundary operator $\partial$, as follows:\\
 $\partial^1 :~ \D\to Hom(\D,\D)$ 
and $\partial^2:~Hom(\D,\D)$$\to$$ Hom(\wedge^2\D,\D).$
\begin{theorem}\label{derivationschar} Let $G$ be a
Lie group, $\G$ its Lie algebra, $T^*G$ its cotangent bundle and $\D:=\G\ltimes \G^*$ the Lie algebra of $T^*G$.
A $1$-cocycle (for the adjoint representation) hence a derivation of $\mathcal D$, has the following form:
\begin{eqnarray}\label{decompositionphi}
\phi(x,f) = \Big(\alpha(x) + \psi(f), \beta(x) + \xi(f)\Big),
\end{eqnarray}
for any $(x,f)$ in $\D$; where $\alpha: \mathcal G \to \mathcal G$ is a derivation of the Lie algebra
$\mathcal G$, $\beta :\mathcal G \to \mathcal G^*$
is a $1$-cocycle of $\mathcal G$ with values in $\mathcal G^*$ for the co-adjoint action of
$\mathcal G$ on $\mathcal G^*$, $\xi : \mathcal G^* \to \mathcal G^*$ and
$\psi : \mathcal G^* \to \mathcal G$ are linear maps satisfying the following conditions:
\beqn
[\xi,ad^*_x] &=& ad^*_{\alpha(x)}, \quad  \forall \; x \in \mathcal G, \label{relation-xi-alpha} \\
\psi \circ ad^*_x &=& ad_x \circ \psi, \quad \forall \; x \in \mathcal G, \label{relation-equivariance} \\
ad^*_{\psi(f)}g &=& ad^*_{\psi(g)}f, \quad \forall \; f,g \in \mathcal G^*.\label{relation-commutation}
\eeqn
\end{theorem}
\begin{proof}
Aiming to get a simpler expression for the derivations, let us write $\phi$ in terms of
its components relative to the decomposition of $\D$
into a direct sum $\D =\mathcal G \oplus  \G^*$ of vector spaces as follows: for all $(x,f)$ in $\D$,
\beq\label{yacine}
\phi(x,f)=\Big(\phi_{11}(x)+\phi_{21}(f),\phi_{12}(x)+\phi_{22}(f)\Big),
\eeq
where $\phi_{11}:\mathcal G \to \mathcal G$, $\phi_{12}: \mathcal G \to
\mathcal G^*$, $\phi_{21}: \mathcal G^* \to \mathcal G$ and
$\phi_{22}:\mathcal G^* \to \mathcal G^*$ are all linear maps. In (\ref{yacine}) we have
made the identifications: $x=(x,0),\; f=(0,f)$ so that $(x,f)$ can also
be written $x+f$. Likewise, we can write
$\phi(x)=(\phi_{11}(x),\phi_{12}(x))$ and $\phi(f)=(\phi_{21}(f),\phi_{22}(f))$,
for any $x$ in $\G$ and any $f$ in $\G^*$; or simply
$\phi(x)=\phi_{11}(x)+\phi_{12}(x)$ and $\phi(f)=\phi_{21}(f)+\phi_{22}(f)$.
In order to find the $\phi_{ij}$'s and hence  all the derivations of $\mathcal D$, we are now going to use the cocycle condition (\ref{suz1}).
For $x,y$ in $\mathcal G \subset \mathcal D$ we have:
\beq\label{suz2}
\phi([x,y]) = \phi_{11}([x,y]) + \phi_{12}([x,y])
\eeq
and
\beqn\label{suz3}
[\phi(x),y] + [x,\phi(y)] &=& [\phi_{11}(x) + \phi_{12}(x),y] + [x,\phi_{11}(y) + \phi_{12}(y)] \nonumber \\
                          &=& [\phi_{11}(x),y] - ad^*_y(\phi_{12}(x)) + [x,\phi_{11}(y)] +
                              ad^*_x(\phi_{12}(y)).
\eeqn
Comparing (\ref{suz2}) and (\ref{suz3}), we first get
$\phi_{11}([x,y]) = [\phi_{11}(x),y] + [x,\phi_{11}(y)]$,
for all  $x,y$ in $\G$. This means that  $\phi_{11}$ is a derivation of the Lie algebra $\mathcal G$.
Secondly, for any $x, y$ in $\G$,
$\phi_{12}([x,y])= ad^*_x(\phi_{12}(y)) - ad^*_y(\phi_{12}(x))$. 
That is $\phi_{12}$ is a $1$-cocycle of $\mathcal G$ with
values on $\mathcal G^*$ for the co-adjoint action of $\mathcal G$ on $\mathcal G^*$.

Now we are going to examine the following case: for all $x$ in $\mathcal G$ and all $f$ in $\mathcal G^*$,
\beqn\label{suz6}
\phi([x,f]) &=& \phi(ad^*_xf)
            = \phi_{21}(ad^*_xf) + \phi_{22}(ad^*_xf)
\eeqn
and
\beqn
[\phi(x),f] + [x,\phi(f)]  & = &  [\phi_{11}(x)+\phi_{12}(x),f]+[x,\phi_{21}(f) + \phi_{22}(f)]  \nonumber \\
                         &=& ad^*_{\phi_{11}(x)}f + [x,\phi_{21}(f)] + ad^*_x\big(\phi_{22}(f)\big).\label{suz7}
\eeqn
Identifying (\ref{suz6}) and (\ref{suz7}) we get on the one hand
 $\phi_{21}(ad^*_xf) = [x,\phi_{21}(f)]
                   = ad_x(\phi_{21}(f)),$
for all $x\in\mathcal G$ and $f\in\mathcal G^*$. We write the above as
$\phi_{21} \circ ad^*_x = ad_x \circ \phi_{21},$
~for all $x$ in $\mathcal G$. That is,  {\em $\phi_{21}:\mathcal G^* \to
  \mathcal G$ is equivariant (commutes) with respect to the adjoint and the co-adjoint actions of
$\mathcal G$
on $\mathcal G$ and $\mathcal G^*$ respectively}.
On the other hand, for all $x$ in $\mathcal G$ and $f$ in $\G^*$, 
$\phi_{22}(ad^*_xf) = ad^*_x(\phi_{22}(f)) + ad^*_{\phi_{11}(x)}f$.
This can be rewritten as
$\phi_{22} \circ ad^*_x - ad^*_x \circ \phi_{22} = ad^*_{\phi_{11}(x)}. $ 
That is, for any element $x$ of $\G$,
 $[\phi_{22},ad^*_x] = ad^*_{\phi_{11}(x)}.$ 
Last, for any $f$ and $g$ in $\G^*,$ we have
\beq\label{suz11}
\phi([f,g]) = 0,
\eeq
and
\beqn\label{suz12}
[\phi(f),g] + [f,\phi(g)] &=& [\phi_{21}(f)+\phi_{22}(f),g]+[f,\phi_{21}(g)+\phi_{22}(g)], \nonumber \\
                          &=& ad^*_{\phi_{21}(f)}g - ad^*_{\phi_{21}(g)}f.
\eeqn
From (\ref{suz11}) and (\ref{suz12}), for all elements $f,g$  of $\G^*$, we have
$ad^*_{\phi_{21}(f)}g = ad^*_{\phi_{21}(g)}f.$

Noting $\alpha :=\phi_{11}$, $\beta:=\phi_{12}$, $\psi:= \phi_{21}$ and $\xi:=\phi_{22}$, we get a proof of Theorem  \ref{derivationschar}.
\end{proof}
\brmq\label{notations}({\bf Notations }) From now on, if $\G$ is a Lie algebra, then \\
(1) $\hbox{\rm der}(\mathcal G)$ will stand for the space of all derivations of $\mathcal G$ ; \\
(2) $\mathcal E$ will stand for the space of linear maps $\xi : \mathcal G^* \to \mathcal G^*$  satisfying Equation (\ref{relation-xi-alpha}), 
for some derivation $\alpha$ of $\mathcal G$;\\
(3) $\G_0\!:=\big\{\phi:\D\to\D,  \phi (x,f)\! =\! (\alpha(x), \xi(f)):\alpha\!\!\in\!\der(\G), \xi\!\in\!\mathcal E, 
[\xi,ad^*_x]=ad^*_{\alpha(x)}, \forall x\in\G\big\};$\\
(4) we may let $\mathcal Q$ stand for the space of 1-cocycles $\beta \!:\! \mathcal G\to \mathcal G^*$  
as above, whereas $\Psi$ may be used for the space of equivariant 
linear $\psi\!:\! \mathcal G^* \to \mathcal G$ as in (\ref{relation-equivariance}), which satisfy (\ref{relation-commutation});\\
(5) we will denote by $\G_1,$ the direct sum  $\G_1:=\mathcal Q\oplus\Psi$ of the vector 
spaces $\mathcal Q$ and $\Psi$.
\ermq
\brmq\label{embeddings} The spaces $\der(\G)$, $\mathcal Q$ and $\Psi$, as in Remark \ref{notations}, are all subsets of $\der(\D),$ 
as follows. A derivation $\alpha$ of $\G$, an equivariant map $\psi\in\Psi,$ a 1-cocycle $\beta \in\mathcal Q$ 
are respectively seen as the elements $ \phi_{\alpha}, \phi_{\psi}, \phi_{\beta}$ of $\der(\D)$ as follows : 
for all $(x,f)$ in $\D$,
$$
\phi_{\alpha}(x,f):=(\alpha(x),-f\circ\alpha)\; ; \;\phi_{\psi}(x,f):=(\psi(f),0)\; ; \;
\phi_{\beta}(x,f):= (0,\beta(x)).
$$
\ermq
\bpro\label{prop:P1}Let $\G$ be a Lie algebra. With the same notations as above, we have: ~(a) if $\G$ has an invertible derivation, 
then so does $T^*\G.$ ~(b) if the Lie algebra $der(\D)$ is unimodular, then $\dim (\Psi)=\dim (\mathcal Q);$ ~(c) $\der(\D)$ is never nilpotent.
\epro
\noindent
%
\begin{proof}
(a) Let $\alpha$ be a linear map $\G\to\G.$ Then $\alpha$ is in $\der(\G)$ if and only if $\phi_{\alpha}$ is in $\der(\mathcal D).$ Moreover, 
$\ker(\phi_{\alpha})=\ker(\alpha)\oplus (Im(\alpha))^o,$ where $(Im(\alpha))^o$ is the subspace of $\G^*$ 
consisting of those $f,$ with $f\circ \alpha=0.$ Hence (a) is proved. For (b) and (c), see Section \ref{chap:bi-inv-endom}.
\end{proof}

\subsection{A structure theorem for the group of automorphisms of $\D$}
\blem\label{structure1}
 The space $\mathcal E,$ as in Remark \ref{notations}, is a Lie algebra.
Namely, if $\xi_1, \xi_2$ in $\mathcal E$ satisfy
$[\xi_i,ad^*_x] = ad^*_{\alpha_i(x)},$ $i=1,2,$
for all $x$ in $\mathcal G$ and some $\alpha_1,\alpha_2$ in $\der(\G)$,
then their Lie bracket $[\xi_1, \xi_2]$ is in $\mathcal E$ and satisfies
$[[\xi_1, \xi_2],ad^*_x] = ad^*_{[\alpha_1,\alpha_2](x)}.$
\elem
\proof
Jacobi's identity in the Lie algebra  $\mathfrak{gl}(\mathcal G^*)$ of endomorphisms
of the vector space $\G^*$, gives,  for any $x$ in $\G$
\beqn
[[\xi_1, \xi_2],ad^*_x]&=& [[\xi_1,ad^*_x], \xi_2] +  [\xi_1, [\xi_2,ad^*_x]]
 = [ad^*_{\alpha_1(x)}, \xi_2] + [\xi_1, ad^*_{\alpha_2(x)}]
 \nonumber\\
 &=& - ad^*_{\alpha_2\circ\alpha_1(x)} + ad^*_{\alpha_1\circ\alpha_2(x)}
 = ad^*_{[\alpha_1,\alpha_2](x)}.\nonumber
~~~~~~~~~~~~~~~\qed \eeqn
\blem \label{structure2}
The space $\G_0,$ as in Remark \ref{notations}, is a Lie subalgebra of
$\der(\D)$.
\elem
\proof This is a consequence of Lemma \ref{structure1}. If $\phi_1:= (\alpha_1, \xi_1)$ and 
$\phi_2:= (\alpha_2, \xi_2)$ are in $\G_0,$ then $[\phi_1,\phi_2]=([\alpha_1,\alpha_2], [\xi_1,\xi_2]).$ 
Indeed, for every $(x,f)\in\D,$ we have
\beqn
 [\phi_1,\phi_2](x,f) & = & \phi_1 \Big( \alpha_2(x),\xi_2(f) \Big) - \phi_2 \Big( \alpha_1(x),\xi_1(f) \Big)
  \nonumber\\
  & = & \Big( \alpha_1\circ\alpha_2(x)~,~ \xi_1\circ\xi_2(f)\Big) - \Big( \alpha_2\circ\alpha_1(x)~,~ \xi_2\circ\xi_1(f) \Big)\nonumber\\
  &=& \Big( [\alpha_1,\alpha_2](x), [\xi_1,\xi_2](f) \Big). \nonumber
~~~~~~~~~~~~~~~~~~~~~~~~~~~~~~~~~~~~~\hfill \qed 
\eeqn
\blem \label{structure3}
Let $\beta\in\mathcal Q$ and $\psi\in\Psi.$ Then
$[\beta,\psi]=(-\psi\circ\beta,\beta\circ\psi)$ belongs to $\G_0$. More precisely
$\beta\circ\psi$ is in $\mathcal E $, $\psi\circ\beta$ is in $\der(\mathcal G)$ and
$[\beta\circ\psi,ad_x^*]=-ad^*_{\psi\circ \beta(x)}$, for any $x$ in $\G$.
\elem
\begin{proof}
First, $\beta$ being a 1-cocycle is equivalent to
\beq
\label{eq:cocycle}
\beta\circ ad_x (y) = ad^*_x \circ \beta (y) -
ad^*_y\circ \beta (x),
\eeq
for all $x,y$ in $\G$. Now for every  $x$ in $\G$ and every $f$ in $\G^*,$ we have
\beqn
[\beta\circ\psi ,ad_x^*](f) &= & \beta\circ\psi \circ ad_x^*(f)-
ad_x^*\circ\beta\circ\psi (f)\nonumber \\
& = & \beta\circ ad_x\circ \psi (f) - ad_x^*\circ\beta\circ\psi(f), ~\text{now take $y=\psi (f)$ in (\ref{eq:cocycle})}\nonumber\\
& = & ad_{x}^*\circ\beta\circ  \psi (f) - ad_{\psi(f)}^*\beta(x) - ad_{x}^*\circ\beta\circ\psi(f),
              \nonumber   \\
& = &  - ad_{\psi(f)}^*\beta(x),~ \text{take $g=\beta(x)$ in
 (\ref{relation-commutation})} \nonumber\\
& = &  - ad_{\psi\circ \beta(x)}^*f = ad_{\alpha(x)}^*f, ~\text{where}~ \alpha= - \psi\circ \beta.\nonumber
\eeqn
Next, the proof that $\psi\circ\beta$ is in $\der(\G)$, is
straightforward. Indeed,  we have
\beqn
\psi\circ\beta [x,y] & = & \psi\Big(ad^*_x\beta (y) - ad^*_y\beta(x)\Big)
                      =  ad_x\circ\psi\circ \beta (y) - ad_y\circ\psi\circ\beta (x)\nonumber\\
                     & = & [x,\psi\circ \beta (y)] +[\psi\circ\beta(x),y],\nonumber
\eeqn
for all vectors $x,y$ in $\G$. Hence $[\beta,\psi]$ belongs to $\G_0,$ for any $\beta$ in $\mathcal Q$  and any  $\psi$ in $\Psi$.
\end{proof}
\blem
Let $\phi:=(\alpha,\xi)$ be in $\G_0$, $\beta: \mathcal G\to \mathcal G^*$  and
$\psi: \mathcal G^* \to \mathcal G$ be respectively in $\mathcal Q$ and $\Psi.$
Then both $[\phi,\beta]$ and $[\phi,\psi]$ are elements of $\G_1$. More precisely
$[\phi,\beta]$ is in $\mathcal Q$ and $[\phi,\psi]$ is in $\Psi.$
Moreover, we have $[\mathcal Q,\mathcal Q]=0$ and $[\Psi,\Psi]=0.$
\elem
\begin{proof}
Let $\phi=(\alpha,\xi)$ be in $\G_0,$ $\beta:\G\to\G^*$ a 1-cocycle and $\psi:\G^*\to\G$ an equivariant linear map. 
Using $\phi_{\beta}$  and $\phi_{\psi}$  as  in Remark \ref{embeddings}, we obtain
\beqn
[\phi,\phi_{\beta}](x,y)&=& \phi\Big(0,\beta(x)\Big)- \phi_{\beta}\Big(\alpha(x),\xi(f)\Big)
= \Big(0,\xi\circ \beta(x)\Big)- \Big(0,\beta\circ\alpha(x)\Big)\nonumber\\
&=& \Big( 0~,~(\xi\circ \beta-\beta\circ\alpha)(x) \Big).\nonumber
\eeqn
Now, let us show that $\tilde\beta\!:=\!\xi\circ \beta\!-\!\beta\circ\alpha\! :\! \G\!\to\!\G^*$ is a $1$-cocycle.
On the one hand, we have 
\beqn\label{eq:beta1}
\xi\circ \beta([x,y])&=&\xi\Big(ad_x^*\beta(y)-ad_y^*\beta(x)\Big)\nonumber\\
&=&\Big([\xi,ad_x^*]+ad^*_x\circ\xi\Big)\big(\beta(y)\big)-\Big([\xi,ad_y^*]
+ad^*_y\circ\xi\Big)\big(\beta(x)\big)\nonumber\\
&=&ad_{\alpha(x)}^*\beta(y)+ad^*_x(\xi\circ\beta(y))-ad_{\alpha(y)}^*\beta(x)-ad^*_y(\xi\circ\beta(x))\nonumber\\
&=&ad^*_x(\xi\circ\beta(y))-ad^*_y(\xi\circ\beta(x))+ad_{\alpha(x)}^*\beta(y)-ad_{\alpha(y)}^*\beta(x).
\eeqn
for all $x,y$ in $\mathcal G$. On the other hand, we also have: for any $x,y$ in $\mathcal G$,
\beqn \label{eq:beta2}
\beta\circ\alpha([x,y])&=& \beta\big([\alpha(x),y]\big)+\beta\big([x,\alpha(y)]\big)\nonumber\\
                       &=& ad_{\alpha(x)}^*\beta(y)-ad^*_y(\beta\circ\alpha(x))+ad^*_x(\beta\circ\alpha(y)) -ad^*_{\alpha(y)} \beta(x).
\eeqn
Subtracting (\ref{eq:beta2}) from (\ref{eq:beta1}), we see that $\tilde\beta[x,y]$ now reads
\beqn \label{eq:beta3}
\tilde\beta[x,y]=  ad^*_x(\xi\circ \beta-\beta\circ\alpha)(y) - ad^*_y(\xi\circ \beta-\beta\circ\alpha)(x)
= ad^*_x\tilde\beta(y)-ad^*_y\tilde\beta(x). \nonumber
\eeqn
Hence $\tilde\beta$ is an element of $\mathcal Q.$ In the same way, we also have: for any $(x,f)$ in $\mathcal D$,
\beqn 
[\phi,\phi_{\psi}](x,f)&=& \phi\big(\psi(f),0\big)- \phi_{\psi}\big(\alpha(x),\xi(f)\big)\nonumber\\
&=& \big(\alpha\circ\psi(f),0\big)- \big(\psi\circ\xi(f),0\big)
= \big( (\alpha\circ\psi- \psi\circ\xi)(f)~,~0 \big). \nonumber
\eeqn
The linear map $\tilde\psi:=\alpha\circ\psi- \psi\circ\xi:\G^*\to\G$ is equivariant,
{\it i.e.} is an element of $\Psi.$ As above, this is seen by first computing, for every elements $x$ of $\G$ and $f$ of $\G^*$,
\beqn \label{eq:psi1}
\alpha\circ\psi(ad_x^*f)=\alpha\big([x,\psi(f)]\big) = [\alpha(x),\psi(f)]+[x,\alpha\circ\psi(f)]
\eeqn
and
\beqn \label{eq:psi2} \psi\circ\xi(ad_x^*f)&=&\psi\circ\big([\xi,ad_x^*]+ad_x^*\circ\xi\big)(f)
=  \psi\big( ad^*_{\alpha(x)}f\big)+\psi\big(ad_x^*\xi(f)\big)\nonumber\\
&=&  \big[\alpha(x),\psi(f)\big]+\big[x,\psi\circ\xi(f)\big],
\eeqn
then subtracting (\ref{eq:psi1}) and (\ref{eq:psi2}).
We have  
$[\phi_{\beta},\phi_{\beta'}](x,f)=\phi_{\beta}(0,\beta'(x))-\phi_{\beta'}(0,\beta(x))=0$ 
and
$[\phi_{\psi},\phi_{\psi'}](x,f)=\phi_{\psi}(\psi'(f),0)-\phi_{\psi'}(\psi(f),0)=0$,
for all $(x,f)$ in $\D.$ In other words, $[\mathcal Q,\mathcal Q]=0$ and $[\Psi,\Psi]=0.$
\end{proof}
We summarize all the above in the
\begin{theorem}\label{structuretheorem}
Let $G$ be a Lie group and $\G$ its Lie algebra. The group Aut($\D$) of
automorphisms of the Lie algebra $\D$ of the cotangent bundle $T^*G$
of $G$, is a super symmetric Lie group. More precisely, its Lie algebra
$\der(\mathcal D)$ is a $\mathbb Z/2\mathbb Z$-graded symmetric (supersymmetric) Lie algebra which 
decomposes into a direct sum of vector spaces
\beq
 \der(\mathcal D):=\mathcal G_0\oplus\mathcal G_1,  \text{~
 with~ } [\mathcal G_i,\mathcal G_j]\subset\mathcal
 G_{i+j}, ~~ i,j \in \mathbb Z/2\mathbb Z = \{0, 1\},
\eeq
where $\mathcal G_0$ and $\mathcal G_1$ are as in Remark \ref{notations}.

Consider the Abelian subalgebras $\tilde \G_1:=\mathcal Q$
and $\tilde \G_1':=\Psi$   of $\der(\D)$ and set $deg(x)=i$, if $x\in\G_i,$ i=1,2.
Then, the  spaces $\G_0\oplus\tilde \G_1$ and $\G_0\oplus\tilde \G_1'$ are subalgebras of \der($\D$)
which are Lie superalgebras, i.e. they are $\mathbb Z/2\mathbb Z$-graded Lie algebras with
the Lie bracket satisfying $[x,y]=-(-1)^{deg(x)deg(y)}[y,x]$ and $[x,[y,z]]=[[x,y],z]+(-1)^{deg(x)deg(y)}[y,[x,z]].$  
\end{theorem}
As a straightforward corollary, we have the following
\begin{proposition}\label{superalgebras} Let $G$ be a
Lie group, $\G$ its Lie algebra, $T^*G$ its cotangent bundle and $\D:=\G\ltimes \G^*$ the Lie algebra of $T^*G$.
Denote by $\Psi$ the space of linear maps $\psi:\G^*\to \G$ satisfying 
$
\psi \circ ad^*_x = ad_x \circ \psi,$ $ \forall \; x \in \mathcal G $ and $
ad^*_{\psi(f)}g = ad^*_{\psi(g)}f, $ $\forall \; f,g \in \mathcal G^*.$
(a)  If $\Psi=\{0\},$ then $\der(\D)$ is a  Lie superalgebra. (b) In particular, if $\G$ is orthogonal with trivial center, e.g. if $\G$ is semi-simple, then $\der(\D)$ is a Lie superalgebra.
\end{proposition}
\proof
(a) With the same notations as in Theorem \ref{structuretheorem}, if $\Psi=\{0\},$  then $\der(\D)=\G_0\oplus\tilde\G_1.$
(b) In this case  $\Psi=\{0\},$ see Proposition \ref{coro:liesuperalgebra} and Section \ref{chap:semi-simple}.\qed

\brmq (a)
In Propositions \ref{decomposition-xi} and \ref{decomposition-xi'}, we will prove that every
element $\xi$ of $\mathcal E$ is the transpose $\xi=(j-\alpha)^t$ of the sum of  an adjoint-invariant endomorphism  $j\in\mathcal J$  
and a derivation $-\alpha$ of $\G.$
~(b) The Lie superalgebras $\G_0\oplus\tilde \G_1$ and $\G_0\oplus\tilde \G_1'$ respectively correspond to the 
subalgebras of all elements of \der($\D$) which preserve the subalgebra $\G$ and the ideal $\G^*$ of $\D$. 
The Lie superalgebra $\G_0\oplus\tilde \G_1'$ can be seen as part of the more general case of derivations of a semi-direct product Lie algebra $\G\ltimes \mathcal N$ which preserve the ideal $\mathcal N$ and which are discussed in \cite{neeb}, amongst other results therein.
\ermq

\subsection{Maps $\xi$ and bi-invariant tensors of type (1,1)}
\subsubsection{Adjoint-invariant endomorphisms}\label{chap:bi-inv-endom}
Linear operators acting on vector fields of a Lie group $G$
can be seen as fields of endomorphisms of its tangent spaces.
Bi-invariant ones correspond to endomorphisms $j:\G \to \G$ of  the
Lie algebra $\G$ of $G$, satisfying $j[x,y]=[jx,y]$, 
for all $x,y$ in $\mathcal G$. We sometimes call such $j$ bi-invariant endomorphisms.
 If we denote by $\nabla$ the connection
on $G$ given on left invariant vector fields by
$\nabla_xy:=\frac{1}{2}[x,y]$, then using the covariant derivative,
we have  $\nabla j=0$, (see e.g. \cite{tondeur}).
As above, let $\mathcal J$ be the space of such $j.$
 Endowed with the bracket $[j,j']:=j\circ j'-j'\circ j,$
the space $\mathcal J$ is a Lie algebra.
 If the dimension of $G$ is even and if in addition $j$ satisfies
$j^2\!\!=~\!\!\!-\!identity$, then $(G,j)$ is a complex Lie group.
Of course, the identity map $id_{\mathcal G}$ of $\G,$  is an element of $\mathcal J.$ We have $[\phi_j,\phi_{id}] =0,$
$[\phi_{\alpha},\phi_{id}] =0,$
$[\phi_{id},\phi_{\psi}] =-\phi_{\psi}$
 and  $[\phi_{id},\phi_{\beta}] = \phi_{\beta},$
 for every $\alpha\in\der(\mathcal G),$ $\psi\in\Psi,$ $\beta\in\mathcal Q$ and $j\in \mathcal J.$ 
 Hence, the adjoint of $\phi_{id}$ is diagonal, with matrix $diag(0,1,-1)$ in the decomposition 
 $\der(\D)=\G_0\oplus\mathcal Q\oplus\Psi.$ Thus trace($ad_{\phi_{id}}$) = $\dim(\mathcal Q)-\dim(\Psi).$  Proposition \ref{prop:P1} (b) is proved. As $\mathcal Q$ contains at least coboundaries  $\beta(x)=ad^*_xf_0$ 
 for some 1-form $f_0$, hence $\mathcal Q\neq 0,$  therefore $ad_{\phi_{id}}$ and thereby $der(T^*\mathcal G)$ 
 are never nilpotent. This proves Proposition \ref{prop:P1} (c).
  Note that, $\Psi=[\phi_{id},\Psi]$ and $\mathcal Q=[\phi_{id},\mathcal Q]$  are subsets of $[\der(\D),\der(\D)],$ hence $ad_{\phi_{\psi}}$ and $ad_{\phi_{\beta}}$ are traceless, $\forall \psi\in\Psi,$ $\forall \beta\in\mathcal Q$.

\subsubsection{ Maps $\xi:\G^*\to\G^*$}
\bpro\label{decomposition-xi} Let $\mathcal G$ be a Lie algebra and $\alpha$ a derivation of
$\mathcal G$. A linear map $\xi':\mathcal G \to \mathcal G$ satisfies
$[\xi',ad_x] = ad_{\alpha(x)}$, for every element $x$ of $\G$, if and only if there
exists a linear map $j:\G\to\G$ satisfying
\beq\label{ameth}
j([x,y]) = [j(x),y] = [x,j(y)],
\eeq
for all $x,y$ in $\G$, such that $\xi'= j + \alpha$.
\epro
\begin{proof}
 Let $\alpha$ be a derivation and $\xi'$ an
endomorphism of $\G$ satisfying the hypothesis of Proposition
\ref{decomposition-xi}, that is,
$[\xi',ad_x] = ad_{\alpha(x)} = [\alpha,ad_x]$, for any $x$ in $\G$.
 We then have,
\beq\label{aby}
[\xi'-\alpha,ad_x] = 0,
\eeq
for any $x$ of $\G$. So the endomorphism $j:=\xi'-\alpha$ commutes with all adjoint operators.
Now a linear map $j:\G\to\G$ commuting with all adjoint operators, satisfies: for all elements $x,y$ of $\G$,
$0=[j,ad_x](y) = j([x,y]) - [x,j(y)]$.
 We also have that for all $x,y$ in $\G$,
$0=[j,ad_y](x) = j([y,x]) - [y,j(x)]$.
Hence, $j([x,y]) = [j(x),y] = [x,j(y)]$, for any $x,y$ in $\G$.
Thus, (\ref{aby}) is equivalent to $\xi'= j +\alpha$, where
$j$ satisfies (\ref{ameth}).
\end{proof}
\bpro \label{decomposition-xi'} Let $\G$ be a nonabelian Lie algebra and $\mathcal S$
the space of endomorphisms  $\xi': \G\to\G$ such that there exists a derivation
$\alpha$ of $\G$ and  $[\xi', ad_x]=ad_{\alpha (x)}$ for  all $x\in \mathcal G$.
1) $ \mathcal S$ is a Lie algebra containing $\mathcal J$ and $\der(\G)$ as subalgebras.
In the case where $\G$ has a trivial centre, then $ \mathcal S$ is  the
semi-direct product $\mathcal S=\der(\G)\ltimes \mathcal J$ of
 $\mathcal J$ and $\der(\G)$.\\
2) The following are equivalent: ~
(a) The linear map $\xi:\G^*\to\G^*$ is an element of  $\mathcal E$ with $\alpha$  as the corresponding derivation 
of $\G$, i.e. $\xi$ satisfies (\ref{relation-xi-alpha}) for the derivation $\alpha$.
~
(b) The transpose $\xi^t$  of $\xi$ is of the form $\xi^t=j-\alpha,$ where $j$ is in $\mathcal J$ and $\alpha$ in $\der(\G)$.
~
(c) $\xi^t$ is an element of $\mathcal S,$ with corresponding derivation $-\alpha.$
The transposition $\xi\mapsto \xi^t$ of linear maps is an anti-isomorphism between the Lie algebras $\mathcal E$ and $\mathcal S.$
\epro
\begin{proof}
1) Using the same argument as in Lemma \ref{structure1}, if for all $x$ in $\G,$ $[\xi'_1, ad_x]=ad_{\alpha_1 (x)}$ 
and $[\xi'_2, ad_x]=ad_{\alpha_2 (x)}$, 
 then $[[\xi'_1,\xi_2'], ad_x]=ad_{[\alpha_1, \alpha_2](x)}$, for any $x$ in $\G$. 
Thus $\mathcal S$ is a Lie algebra. From Proposition \ref{decomposition-xi}, 
there exist $j_i$ in $\mathcal J$ such that $\xi'_i=\alpha_i+j_i ,$ $i=1,2.$
Obviously, $\mathcal S$ contains $\mathcal J$ and $\der(\G).$
Thus, as a vector space, $\mathcal S$ decomposes as $\mathcal S= \der(\G) +\mathcal J.$
Now, the Lie bracket in $\mathcal S$ reads
\beq \label{eq:liebraket1-S}
[\xi'_1,\xi'_2]= [\alpha_1+j_1, \alpha_2+j_2]= [\alpha_1, \alpha_2] +[\alpha_1,j_2]+[j_1, \alpha_2]+[j_1,j_2].
\eeq
Of course, $[\alpha_1, \alpha_2]$ is in $\der(\G)$. From Section \ref{chap:bi-inv-endom}, we know that $\mathcal J$ is a Lie algebra, 
hence $[j_1,j_2]$ is in $\mathcal J$. It is easy to check that
$[\alpha,j]\in \mathcal J,$
for all $\alpha$ in $\der(\G)$ and for all $j$ in $\mathcal J$. Indeed, the following holds
\beqn
[\alpha,j]([x,y])&=& \alpha \big([j(x),y]\big) - j\big([\alpha(x),y]+[x,\alpha(y)]\big)\nonumber \\
& = &  [\alpha\circ j(x),y]+ [j(x),\alpha(y)] - [j\circ \alpha(x),y]-[j(x),\alpha(y)]
=  [[\alpha,j](x),y],\nonumber
\eeqn
for all $x,y$ in $\G$.
The intersection $\der(\G)\cap \mathcal J$ is made of elements $j$ of $\mathcal J$ whose image Im($j$) is a subset of 
the centre $Z(\G)$ of $\G.$
Hence if $Z(\G)=0$, then $\mathcal S=$\der($\G$)$ \oplus\mathcal J$ and as a Lie algebra,   
$\mathcal S= $ $\der(\G)\ltimes \mathcal J.$ Using this  decomposition, we can also rewrite (\ref{eq:liebraket1-S}) as
\beq \label{eq:liebraket2-S}
[\xi'_1,\xi'_2]= \left[(\alpha_1,j_1)~,~ (\alpha_2,j_2)\right]= \left([\alpha_1, \alpha_2]~,~    [j_1,j_2] +[\alpha_1,j_2]+[j_1, \alpha_2]\right).
\eeq
The equivalence between (b) and (c) comes directly from Proposition \ref{decomposition-xi}.
 Now let $\xi\in\mathcal E$, with $[\xi, ad^*_x] = ad^*_{\alpha (x)}$, $\alpha \in \der(\G)$, then
$- ad_{\alpha (x)} = [\xi, ad^*_x]^t = - [\xi^t, (ad^*_x)^t] = [\xi^t, ad_x]$. 
Hence $\xi^t\in \mathcal S$, with $ad_{\alpha' (x)}  = [\xi^t, ad_x]$, for all $x\in \G$,
where $\alpha':=-\alpha$. Thus, (a) implies (c). From Proposition \ref{decomposition-xi}, there exists $j\in\mathcal J$ 
such that $\xi^t=-\alpha+j.$ Now it is straightforward that if  (b) $\xi^t=-\alpha+j$ with $\alpha$ in $\der(\G)$ and $j$ in $\mathcal J$, 
then $\xi$ satisfies $[\xi,ad^*_x]=ad^*_{\alpha(x)}$, for all $x$ in $\G$. Hence (c) implies (a).
Of course, we also know that $[\xi_1,\xi_2]^t= - [\xi_1^t,\xi_2^t]$, for every $\xi_1,\xi_2\in\mathcal E$.
\end{proof}
\begin{lemma}\label{lem:decomp-semi-simpl} Let $\xi': \mathcal G \to \mathcal G$ be a linear map
such that there exists $\alpha : \mathcal G \to \mathcal G$ linear and $[\xi',ad_x]=ad_{\alpha(x)}$,
for all $x$ in $\mathcal G$. Then $\xi'$ preserves every ideal $\mathcal A$ of $\mathcal G$ satisfying
$[\mathcal A,\mathcal A]=\mathcal A.$  In particular, if $\mathcal G$ is semi-simple and
$ \G = \s_1\oplus\s_2\oplus\cdots \oplus \s_p$ is a decomposition of $\mathcal G$ into a sum of simple ideals $\s_1,\ldots , \s_p$, 
then $\xi'(\s_i)\subset \s_i$, for $i=  1,\ldots, p.$
\end{lemma}
\begin{proof}
The proof is straightforward. Indeed, every element $x$ of an ideal $\mathcal A$ satisfying the hypothesis of 
Lemma \ref{lem:decomp-semi-simpl}, is a finite sum of the form $x=\displaystyle \sum_i [x_i,y_i]$ where $x_i,y_i$ 
are all elements of $\mathcal A$.  But as $\mathcal A$ is an ideal,
\beqn
\xi' ([x_i,y_i])&=&\xi' \circ ad_{x_i} (y_i)=\big([\xi',ad_{x_i}] + ad_{x_i}\circ \xi' \big)(y_i)\nonumber\\
&=&\big(ad_{\alpha(x_i)} + ad_{x_i}\circ \xi'\big)(y_i) =[\alpha(x_i),y_i] +
[x_i, \xi'(y_i)]\nonumber
\eeqn
is again an element of $\mathcal A$. Hence we have $\xi' (x) = \displaystyle \sum_i \big([\alpha(x_i),y_i] +
[x_i, \xi'(y_i)]\big)$ is in $\mathcal A$.
\end{proof}

\subsection{Equivariant maps $\psi: \G^*\to\G $}\label{chap:equivariance}
\blem\label{lem:intertwining}
Let $\G$ be a Lie algebra and $\psi$ an element of $\Psi$.
Then, \\
(a) $\hbox{\rm Im} \psi$ is an Abelian ideal of $\G$ and we have $\psi(ad_{\psi (g)}^*f)=0$,
for every $f,g$ in $\G^*$; \\
(b) $\psi$ sends closed forms on $\G$ in the centre of $\G$;\\
(c) $[Im \psi,\G] \subset \ker f$, for all $f$ in $\ker \psi$;\\
(d) the map $\psi$ cannot be invertible if $\G$ is not Abelian.
\elem
\begin{proof}
(a) We have, 
$[\psi(f),x]  = -(ad_x \circ \psi)(f) =  -(\psi \circ ad_x^* )(f) \in Im \psi$,
for every elements $f$ in $\G^*$ and $x$ in $\G$. Hence $Im \psi$ is an ideal of $\G$.
 Now, for every $f,g$ in $\G^*$, since $\psi(f)$ and $\psi(g)$ are elements of $\G$,  we also have
$\psi \circ ad^*_{\psi(f)} = ad_{\psi(f)} \circ \psi$ and
$\psi \circ ad^*_{\psi(g)} = ad_{\psi(g)} \circ \psi$. On the one hand,
$(\psi \circ ad^*_{\psi(f)})(g) =  (ad_{\psi(f)} \circ \psi)(g)$, that is,
$\psi(ad^*_{\psi(f)}g) = [\psi(f),\psi(g)]$. 
On the other hand, $(\psi \circ ad^*_{\psi(g)})(f) =  (ad_{\psi(g)} \circ \psi)(f)$ or equivalently
$\psi(ad^*_{\psi(g)}f) = [\psi(g),\psi(f)]$. 
From the latter equations, we get
$[\psi(f),\psi(g)]\!=\!\psi(ad^*_{\psi(f)}g)\!=\!\psi(ad^*_{\psi(g)}f)\! =\! [\psi(g),\psi(f)]$.
The latter implies that
$[\psi(f),\psi(g)]\!=\!\psi(ad^*_{\psi(g)}f)\!\! =\! 0,$ 
 for all   $f,g$ in $\G^*$. So  (a) is proved.

 (b) Let $f$ be a closed form on $\G$, that is, $f$ is in $\G^*$ and $ad_x^*f=0$, for all $x$ in $\G$.
The relation (\ref{relation-commutation}) implies that $ad^*_{\psi(f)}g=0 $,
for any $g$ in $\G^*$. Thus, for any $y$ in $\G$ and  $g$ in $\G^*$, we have 
$g([\psi(f),y])=0, $
 and hence $[\psi(f),y]=0$. Thus $\psi(f)$ belongs to the centre of $\G$.

(c) If $f \in \ker \psi$, then $ad^*_{\psi(g)}f = ad^*_{\psi(f)}g=0$, for any $g$ in $\G^*$,
or equivalently, for any $x$ in $\G$ and $g$ in $\G^*$, $f([\psi(g),x])=0.$
Then $[Im \psi,\G] \subset \ker f$, for every  $f$ of $\ker \psi$.
 
(d) From (a), the map $\psi$ satisfies $\psi(ad^*_{\psi(g)}f)= 0$, for any $f,g$ in $\G^*$.
 There are two possibilities here: ~
 (i) either there exist $ f,g$  in $\G^*$ such that $ad^*_{\psi(g)}f\neq 0$, in which case $ad^*_{\psi(g)}f$ 
 belongs to $\ker\psi\ne 0$ and thus  $\psi$ is not invertible;
~(ii) or else, suppose $ad^*_{\psi(g)}f = 0$, for all $f,g$ in $\G^*$.
This implies that $\psi(g)$ belongs to the centre of $\G$ for every
$g$ in $\G^*$. In other words, the centre of $\G$ contains $Im \psi$.
But since $\G$ is not Abelian, the centre of $\G$ is different from $\G$,
hence $\psi$ is not invertible.
\end{proof}
\blem\label{lem:coad-inv}
The space of equivariant maps $\psi:\G^* \to \G$ bijectively corresponds to that of
$\G$-invariant bilinear forms on the $\G$-module $\G^*$ for the co-adjoint representation.
\elem
\begin{proof}
Indeed, each such $\psi$ defines a unique co-adjoint-invariant bilinear form $\langle,\rangle_{\psi}$ on $\G^*$ as follows
$\langle f,g\rangle_{\psi}:= \langle \psi (f), g\rangle,$
for all $f,g$ in $\G^*$, where the right hand side is the duality pairing
$\langle f,x \rangle = f(x)$, $x$ in $\G$, $f$ in $\G^*,$ as above.
The co-adjoint-invariance reads
$\langle ad_x^* f,g\rangle_{\psi}+ \langle f,ad_x^* g \rangle_{\psi}= 0,$
for all $x$ in $\G$ and all $f,g$ in $\G^*$; and is due to the simple equalities
$\langle ad_x ^*f,g \rangle_{\psi} = \langle \psi(ad_x^*f),
g\rangle = \langle ad_x\psi(f),g \rangle$
 $=- \langle \psi(f),ad_x^*g \rangle=-\langle f,ad_x^*g\rangle_{\psi}. $ 
Conversely, every $\G$-invariant bilinear form $\langle,\rangle_1$ on $\G^*$ gives rise to a unique
linear map $\psi_1: \G^*\to \G$ which is equivariant with respect to the adjoint and co-adjoint representations of $\G$, by the formula
$\langle\psi_1(f),g\rangle:=\langle f,g\rangle_1.$
\end{proof}
If $\psi$ is symmetric or skew-symmetric with right to the duality pairing, then so is $\langle,\rangle_\psi$ and vice versa. Otherwise,
$\langle,\rangle_\psi$ can be decomposed into a symmetric and a skew-symmetric parts $\langle , \rangle_{\psi,s}$ and  $\langle ,
\rangle_{\psi,a}$ respectively, defined by the following formulas:
\beqn
\langle f, g \rangle_{\psi,s} :=\frac{1}{2}\Big[\langle f,g\rangle_\psi +\langle g, f\rangle_\psi \Big], ~ ~
\langle f, g \rangle_{\psi,a} := \frac{1}{2}\Big[ \langle f, g \rangle_\psi - \langle g, f\rangle_\psi \Big].
\eeqn
The symmetric and skew-symmetric parts $\langle ,  \rangle_{1,s}$ and $\langle , \rangle_{1,a}$ of a $\G$-invariant bilinear form 
$\langle ,\rangle_{1}$, are also $\G$-invariant. 
The radical $Rad\langle ,  \rangle_{1}:=\{ f\in \G^*, \langle f,g  \rangle_{1}=0, \forall g\in \G^*\}$ of the form 
$\langle ,  \rangle_{1},$ contains the co-adjoint orbits of all its points (see a remark in \cite[p. 2297]{me-re93}).

\subsection{Cocycles $\beta: \G\to\G^*$} \label{chap:cocyles}
The 1-cocycles for the co-adjoint representation of a Lie algebra $\G$ over $\K$ are linear maps
$\beta:\G\to \G^*$ satisfying the condition
$\beta ([x,y])=ad_x^*\beta (y) - ad^*_y\beta (x)$, for every elements $x,y$ of $\G.$
To any given 1-cocycle $\beta$, corresponds a bilinear form $\Omega_\beta$ on $\G$, by the formula
\beq \label{eq:cocycle-2-forms}
\Omega_\beta (x,y):=\langle \beta (x),y\rangle,
\eeq
for all $x,y$ in $\G$, where $\langle,\rangle$ is again the duality pairing between
elements of $\G$ and $\G^*$.

The bilinear form $\Omega_\beta$  is skew-symmetric (resp. symmetric, non-degenerate)
if and only if $\beta$ is skew-symmetric (resp. symmetric, invertible).
Skew-symmetric such cocycles $\beta$ are in bijective
correspondence with closed 2-forms in $\mathcal G$, via the formula (\ref{eq:cocycle-2-forms}). In this sense, 
the cohomology space $H^1(\D,\D)$  contains the second cohomology space $H^2(\G,\K)$ of $\G$ with coefficients in $\K$ 
for the trivial action of $\G$ on $\K$. Hence, $H^1(\D,\D)$ somehow contains the second space $H^2_{inv}(G,\K)$ 
of left invariant de Rham cohomology  $H^*_{inv}(G,\K)$ of any Lie group $G$ with Lie algebra $\G.$
Invertible skew-symmetric ones, when they exist, are those giving rise to symplectic forms or
equivalently to invertible solutions of the Classical Yang-Baxter
Equation. The study and classification of the solutions of the Classical Yang-Baxter Equation is a still open 
problem in Geometry, Theory of integrable systems. In Geometry, they give rise to very interesting structures in the framework 
of Symplectic Geometry, Affine Geometry, Theory of
Homogeneous K\"ahler domains,  (see e.g. \cite{di-me-cybe} and references therein).

If $\G$ is semi-simple, then every cocycle $\beta$ is a coboundary, that is, there exists
$f_\beta$ in $\G^*$ such that $\beta (x)= -ad^*_xf_\beta$, for any $x$ in $\G.$

\subsection{Cohomology space $H^1(\D,\D)$} \label{chap:cohomology}
In  Theorem \ref{thm:specialcase} below, we show that the first cohomology space $H^1(\D,\D)$ of the Chevalley-Eilenberg cohomology 
associated with the adjoint action of $\D$ on itself, is isomorphic to 
$H^1(\G,\G)\oplus  \mathcal J^t\oplus H^1(\G,\G^*)\oplus \Psi,$  where $H^1(\G,\G)$ and $H^1(\G,\G^*)$ 
are the first cohomology spaces associated with the adjoint and co-adjoint actions of $\G$, respectively; and
$\mathcal J^t:=\{j^t, j \in \mathcal J \}$ (space of transposes of elements of $\mathcal J$).

\begin{theorem}\label{thm:specialcase}
The isomorphism
\beq \label{eq:cohomology} \Phi: \der(\G)\oplus\mathcal J^t\oplus \mathcal Q\oplus \Psi\to \der(\D); ~~(\alpha,j^t,\beta,\psi)\mapsto 
\phi_\alpha+\phi_j+\phi_\beta+\phi_\psi,
\eeq
between the vector spaces  $\der(\G)\oplus\mathcal J^t\oplus \mathcal Q\oplus \Psi$ and $\der(\D)$, induces an isomorphism   
$\bar \Phi$ in cohomology, between the spaces $H^1(\G,\G)\oplus \mathcal J^t\oplus H^1(\G,\G^*)\oplus \Psi$ and $H^1(\D,\D).$
~(a) If $\G$ is semi-simple, then  $\Psi=\{0\}$ and thus $H^1(\D,\D) \stackrel{\sim}{=}  \mathcal J^t.$ Moreover if $\K=\C$, we have
$\mathcal J \stackrel{\sim}{=} \C^p$,
where $p$ is the number of simple
 ideals $\s_i$ of $\G$ such that $\G=\s_1\oplus\cdots\oplus \s_p$. Hence, of course, $H^1(\D,\D) \stackrel{\sim}{=}   \C^p.$
~(b) If $\G$ is a compact Lie algebra, with centre $Z(\G)$ of dimension $k$, we get $ H^1(\G,\G) \stackrel{\sim}{=} \hbox{\rm End}(Z(\G)),$ 
$ H^1(\G,\G^*) \stackrel{\sim}{=} L(Z(\G), Z(\G)^*),$ $\Psi\stackrel{\sim}{=} L(Z(\G)^*, Z(\G)).$ Hence,
we get $ H^1(\D,\D) \stackrel{\sim}{=} (\hbox{\rm End}(\K^k) )^3\oplus \mathcal J^t$.
Here, if $E,F$ are vector spaces, $L(E,F)$ is the space of linear maps $E\to F$ and $L(E,E):=\hbox{\rm End}(E).$ 
If $\K=\C,$ then $ \mathcal J  \stackrel{\sim}{=} \C^p\oplus \hbox{\rm End}(Z(\G)),$ and 
$ H^1(\D,\D) \stackrel{\sim}{=} (\hbox{\rm End}(\C^k) )^4\oplus \C^p,$ where  $p$ is the number of the simple
components (ideals) of the derived ideal $[\G,\G]$ of $\G$.
\end{theorem}
\begin{proof}
Following Remarks \ref{notations} and  \ref{embeddings}, we can embed $\der(\G)$ as a subalgebra $\der(\G)_1$ of $\der(\D)$, 
using the linear map $\alpha\mapsto \phi_\alpha,$ with $\phi_\alpha(x,f)=(\alpha(x),-f\circ\alpha)$. In the same way, we have 
constructed $\mathcal Q$ and $\Psi$ as subspaces of  $\der(\D).$ Likewise, $\mathcal J^t:=\{j^t,  j\in \mathcal J\}$
is seen as a subspace of $\der(\D),$ via the linear map $j^t\mapsto\phi_j$, with $\phi_j(x,f)=(0,f\circ j).$
Now an exact derivation of $\D$, {\it i.e.} a $1$-coboundary for the Chevalley-Eilenberg cohomology associated with the adjoint 
action of $\D$ on itself, is of the form $\phi_0=\partial v_0=ad_{v_0}$ for some element $v_0:=(x_0,f_0)$ of the $\D$-module $\D.$ 
That is, $\phi_0(x,f)=(\alpha_0(x),\beta_0(x)+\xi_0(f))$, where $\alpha_0(x)=[x_0,x]$, $\beta_0(x)=-ad^*_xf_0$,  
$\xi_0(f)=ad^*_{x_0}f$. As we can see $\phi_0=\phi_{\alpha_0}+\phi_{\beta_0}=\Phi(\alpha_0,0,\beta_0,0)$ and 
the linear map $\Phi$ in (\ref{eq:cohomology}) induces an isomorphism   $\bar \Phi$ in cohomology, 
between the spaces $H^1(\G,\G)\oplus \mathcal J^t\oplus H^1(\G,\G^*)\oplus \Psi$ and $H^1(\D,\D).$
The isomorphism in cohomology simply reads 
$\bar \Phi (class(\alpha), j^t,class(\beta),\psi)=class(\phi_\alpha+\phi_j+\phi_\beta+\phi_\psi).$

The proof of the rest of Theorem \ref{thm:specialcase} is given by different lemmas and propositions 
discussed in Section \ref{chap:orthogonal-algebra}.
\end{proof}
Let us remark that $\Phi (\der(\G)\oplus\mathcal J^t)= \der(\G)_1\oplus\mathcal J^t=\G_0$ 
and  $\Phi (\mathcal Q\oplus\Psi)=\G_1$.

\section{Case of orthogonal Lie algebras}\label{chap:orthogonal-algebra}
In this section, we prove that if a Lie algebra $\G$ is orthogonal, then $\der(\G)$ and $\mathcal J$  
completely characterize the Lie algebra $\der(\D)$, and hence the group of 
automorphisms of the cotangent bundle of any connected Lie group with Lie algebra $\G.$ We also show that $\mathcal J$ is
isomorphic to the space of adjoint-invariant bilinear forms on $\G.$

\subsection{General orthogonal Lie algebras}
Let $(\G,\mu)$ be an orthogonal Lie algebra and consider the isomorphism $\theta : \G \to \G^*$ of $\G$-modules,  
given by $\langle \theta(x),y \rangle := \mu(x,y)$, as in Section \ref{chap:preliminaries}.
Of course, $\theta^{-1}$ is an equivariant map. But if $\G$ is not Abelian, invertible equivariant linear maps do not 
contribute to the space of derivations of $\D,$ as discussed in Lemma~\ref{lem:intertwining}.
We pull co-adjoint-invariant bilinear forms $B'$ on $\G^*$ back to adjoint-invariant bilinear
forms on $\G,$ as follows $B(x,y):= B'(\theta(x),\theta(y))$. 
Indeed, we have, for all $x,y,z$ in $\G,$ ~\\
$ B([x,y],z)= B'(\theta([x,y]),\theta(z))= B'(ad^*_x\theta(y),\theta(z))$
$= - B'(\theta(y),ad^*_x\theta(z))=-B(y,[x,z])$. 
\bpro\label{prop:ortho1}
If $(\G,\mu)$ is an orthogonal Lie algebra, then there is an isomorphism between any two of the following vector spaces:~
 (a) the space $\mathcal J$ of linear maps $j:\G\to\G$ satisfying $j[x,y]=[jx,y]$, for every
 $x,y$ in $\G$; ~
 (b)  the space of linear maps $\psi:\G^*\to\G$ which are equivariant with respect to the co-adjoint and the adjoint representations of $\G$;
 ~
 (c) the space of bilinear forms $B$ on $\G$ which are adjoint-invariant, {\it i.e.}~ 
 $B([x,y],z)+B(y,[x,z])=0, $ ~
for all $ x,y,z$ in $\G$;
~
 (d) the space of bilinear forms $B'$ on $\G^*$ which are co-adjoint-invariant, {\it i.e.} ~
 $B'(ad^*_xf,g)+B'(f,ad^*_xg)=0,$~
for all $x$ in $\G$,   $f,g$ in $\G^*$.
\epro
\begin{proof}
$\bullet$ The linear map $\psi\mapsto \psi\circ\theta$ is an isomorphism between the space of equivariant linear 
maps $\psi:\G^*\to\G$ and the space $\mathcal J.$
Indeed, if $\psi$ is equivariant, we have, for all $x,y$ in $\mathcal G$,
$\psi\circ\theta([x,y])=-\psi(ad^*_y\theta(x))=-ad_y\psi(\theta(x))=[ \psi\circ\theta (x),y]$. 
 Hence, $\psi\circ\theta$ is in $\mathcal J.$ Conversely, if $j$ is in $\mathcal J$, then $j\circ\theta^{-1}$ is equivariant, as it satisfies
  $j\circ\theta^{-1}\circ ad^*_x= j\circ ad_x\circ\theta^{-1}= ad_x \circ j\circ\theta^{-1}$. 
   This correspondence is obviously linear and invertible. Hence, we get the isomorphism between (a) and (b).
\newline\noindent
$\bullet$  The isomorphism between the space $\mathcal J$ and adjoint-invariant bilinear forms is given as follows:
 $j\in\mathcal J\mapsto B_j$,  where $B_j(x,y):=\mu(j(x),y).$ 
for any $x,y$ in $\G$. For any $x,y,z$ in $\G$
$B_j([x,y],z):=\mu(j([x,y]),z)=\mu([x,j(y)],z)= - \mu(j(y),[x,z])= -B_j(y,[x,z])$. 
Conversely, if $B$ is an adjoint-invariant bilinear form on $\G,$ then the endomorphism $j$, defined by 
 $\mu(j(x),y):=B(x,y)$, belongs to $\mathcal J,$ as it satisfies
$\mu(j([x,y]),z):= B([x,y],z)=B(x,[y,z])= \mu(j(x) ,[y,z]) = \mu([j(x),y] ,z),$ 
for all elements  $x,y,z$ of $\G$.
\\
$\bullet$ From Lemma \ref{lem:coad-inv}, the space of equivariant linear maps $\psi$ bijectively corresponds to that of 
co-adjoint-invariant bilinear forms on $\G^*,$ via $\psi\mapsto \langle , \rangle_\psi$.
\end{proof}
Now, suppose $\psi$ is skew-symmetric with respect to the duality pairing.
 Let $\omega_\psi$ denote the corresponding skew-symmetric bilinear form  on $\G$, i.e.  
$\omega_\psi(x,y) := \mu(\psi\circ\theta (x),y),$
for all $x,y$ in $\G$. Then, $\omega_\psi$ is adjoint-invariant.  If we denote by $\partial$  the Chevalley-Eilenberg 
coboundary operator, that is,
$- (\partial\omega_\psi)(x,y,z) =  \omega_\psi([x,y],z)+\omega_\psi([y,z],x)+\omega_\psi([z,x],y),$
the following formula holds true
$(\partial \omega_\psi)(x,y,z)
= -\omega_\psi([x,y],z),$ 
for all $x,y,z$ in $\G$.
\bcor
The following are equivalent:~
(a) $\omega_\psi$ is closed;~
(b) $\psi\circ \theta([x,y])=0$, for all $x,y$ in $\G$;~
(c) $Im \psi$ is in the centre of $\G$.\\
\noindent
In particular, if $\dim [\G,\G]\ge \dim\G-1$, then $\omega_\psi$ is closed if and only if $\psi=0.$
\ecor
\begin{proof}
 The above equality also reads: for all  $x,y,z$ in $\G$,
\beq \label{partial-omega}
\partial\omega_\psi(x,y,z)=-\omega_\psi([x,y],z)=-\mu(\psi\circ\theta([x,y]),z),
\eeq
 and gives the proof that (a) and (b) are equivalent. In particular,
if $\G=[\G,\G]$ then, obviously  $\partial\omega_\psi =0$ if and only if $\psi=0,$ as $\theta$  is invertible.
Now  suppose $\dim [\G,\G]= \dim\G-1$ and set $\G=\R x_0\oplus [\G,\G],$ for some $x_0$ in  $\G.$
If $\omega_\psi$  is closed, we already know that $\psi\circ \theta$ vanishes on $[\G,\G].$ Below, we show that, 
it also does on $\R x_0.$ Indeed, the formula
$0=-\omega_\psi([x,y],x_0)=\omega_\psi(x_0, [x,y])$$=\mu(\psi\circ\theta(x_0), [x,y]),$
for all  $x,y$ in $\G$, obtained by taking $z=x_0$ in (\ref{partial-omega}),
coupled with the obvious equality $0=\omega_\psi(x_0,x_0)= \mu(\psi\circ\theta(x_0), x_0)$,
are equivalent to $\psi\circ\theta(x_0)$ satisfying $\mu(\psi\circ\theta(x_0),x)=0$
for all $x$ in $\G$. As $\mu$ is non-degenerate, this means that $\psi\circ\theta(x_0)=0$.
Hence $\psi=0.$ Now, as every $f$ in $\G^*$ is of the form $f=\theta(y),$ 
for some $y$ in $\G,$  the formula
$\psi\circ\theta([x,y])= \psi\circ ad_x^*\theta(y)=  ad_x\circ \psi\circ \theta(y) = [x, \psi\circ \theta(y)],$
for all $x,y$ in $\G$, shows that $\psi\circ\theta([x,y])=0$, for all $x,y$ of $\G$ if and only if $Im \psi$ 
is a subset of the centre of $\G.$ Thus, (b) is equivalent to (c).
\end{proof}
Now we pull any  $\xi$ of $\mathcal E$ back to an endomorphism $\xi'$ of $\G$ by the formula $\xi':=\theta^{-1} \circ
\xi \circ \theta.$
\bpro\label{guimbita1}
Let $(\G,\mu)$ be an orthogonal Lie algebra and $\G^*$ its dual space. Define $\theta : \G \to \G^*$ by 
$\langle \theta(x),y \rangle := \mu(x,y)$,
and let $\mathcal E$ and $\mathcal S$ stand for the same Lie algebras as above.
The linear map $U:~\xi \mapsto \xi':=\theta^{-1} \circ
\xi \circ \theta$ is an isomorphism of Lie algebras between $\mathcal E$ and $\mathcal S.$
\epro
\proof
Let $\xi$ be in $\mathcal E,$ with $[\xi,ad^*_x]=ad^*_{\alpha(x)}$, for every $x$ in $\G$.
The image $U(\xi)=:\xi'$ of $\xi,$ satisfies, for any  $x$ in $\G$,
\beqn
[\xi',ad_x] &:=& \xi' \circ ad_x - ad_x \circ \xi' 
            = \theta^{-1} \circ \xi \circ \theta \circ ad_x - ad_x \circ \theta^{-1} \circ \xi \circ \theta \nonumber \\
            &=& \theta^{-1} \circ \xi \circ ad^*_x \circ \theta - \theta^{-1} \circ ad^*_x \circ \xi \circ \theta \nonumber \\
     &=& \theta^{-1}\circ (\xi \circ ad^*_x - ad^*_x \circ \xi)\circ \theta 
     = \theta^{-1} \circ ad^*_{\alpha(x)} \circ \theta, \qquad \mbox{ since } [\xi,ad^*_x] = ad^*_{\alpha(x)} ; \nonumber \\
     &=& \theta^{-1} \circ \theta\circ ad_{\alpha(x)}, \qquad \mbox{ since $\theta$ is equivariant } ; \nonumber \\
    &=& ad_{\alpha(x)}.   \nonumber
\eeqn
Now we have $[U(\xi_1),U(\xi_2)]=U([\xi_1,\xi_2])$ for all $\xi_1,\xi_2$ in $\mathcal E$, as seen below.
\beqn
[U(\xi_1),U(\xi_2)]&:=&U(\xi_1)U(\xi_2) - U(\xi_2)U(\xi_1)\nonumber\\
&:=&\theta^{-1}\circ\xi_1\circ\theta\circ \theta^{-1}\circ\xi_2\circ\theta - \theta^{-1}\circ\xi_2\circ\theta\circ \theta^{-1}\circ\xi_1\circ\theta\nonumber\\
&=& \theta^{-1}\circ[\xi_1,\xi_2]\circ\theta= U([\xi_1,\xi_2]).\nonumber ~~~~~~~~~~~~~~~~~~~~~~~~~~~~~~~~~~~~~~~~~~~\qed
\eeqn
\bpro
The linear map $P:~\beta\mapsto D_\beta:=\theta^{-1}\circ \beta,$ is an isomorphism between the space of cocycles 
$\beta:\G\to\G^*$ and the space $\der(\G)$ of derivations of $\G.$
\epro
\begin{proof}
  The proof is straightforward. If $\beta:\G\to\G^*$ is a cocycle, then the linear map 
  $D_\beta:\G\to\G$,~ $x\mapsto\theta^{-1}(\beta(x))$ is a derivation of $\G,$ as we have
$D_\beta[x,y]=\theta^{-1}\big(ad^*_x\beta(y)-ad^*_y\beta(x)\big)=[x,\theta^{-1}(\beta(y))]- [y,\theta^{-1}(\beta(x))].$ 
 Conversely, if $D$ is a derivation of $\G,$ then the linear map \\
$\beta_D:=P^{-1}(D)=\theta\circ D:\G\to\G^*,$ is 1-cocycle. Indeed, for every $x,y$ in $\G,$ we have
\\
$\beta_D[x,y]=\theta([Dx,y]+[x,Dy])= - ad^*_y(\theta\circ D (x)) + ad^*_x (\theta\circ D (y)).$
\end{proof}
It is now easy to see the following straightforward corollary.
\begin{proposition}\label{coro:liesuperalgebra} If $\G$ is an orthogonal Lie algebra, any derivation of $\D:=T^*\G$ has the following form:
for any element $(x,f)$ in $T^*\G$, 
\begin{eqnarray}\label{isomorphismbinvendo}
\phi(x,f)=\Big(\alpha_1(x)+j_2\circ\theta^{-1}(f)\;,\; \theta\circ\alpha_2(x)+(j_1-\alpha_1)^t(f)\Big),
\end{eqnarray}
where $\alpha_1,\alpha_2\in\der(\G)$, ~$j_1,j_2\in\mathcal J$ and   $Im(j_2)$ is a subset of the center $Z(\G)$ of $\G.$
\newline\noindent In particular, the map  $\{j\in\mathcal J,$ such that $Im(j)\subset Z(G)\}$ $\to$ $\Psi,$  $j\mapsto j\circ\theta^{-1}$ is an isomorphism.
\end{proposition}
\proof  Equality (\ref{isomorphismbinvendo}) is already proved above. The condition  $ad_{j_2\circ\theta^{-1}(f)}^*g\!\!=\!\!ad_{j_2\circ\theta^{-1}(g)}^*\!f$, for all $f,g$ in~$\G^*$ required for elements $j_1\circ\theta^{-1},$ $  j_2\circ\theta^{-1}$ of $\Psi,$ is equivalent to 
 $\theta([j_2(x),y])=ad_{j_2(x)}^*\theta(y)\!\!=\!\!ad_{j_2(y)}^*\!\theta(x)= -\theta([x,j_2(y)])=-\theta([j_2(x),y]),$ for all $x,y\in\G$.  Hence $\theta([j_2(x),y])=0,$ that is, $[j_2(x),y]=0,$  for all $x,y\in\G$. Proposition \ref{prop:ortho1} completes the proof.\qed

\subsection{Case of semi-simple Lie algebras}\label{chap:semi-simple}
If $\mathcal G$ is a semi-simple Lie algebra and
$ \G = \s_1\oplus\cdots \oplus \s_p$  a decomposition of $\mathcal G$ into a direct sum of simple ideals,
let us denote by $\s_i^*,$ the dual vector space to each simple component $\s_i,$ $i=1,\ldots, p.$ We view each $\s_i^*$ 
as the subspace of $\G^*$ made of all those linear forms on $\G$ whose kernel contains all the $\s_j$, $j\neq i,$ 
except possibly $\s_i.$ This induces a decomposition of $\G^*$ into the following direct sum 
$ \G^* = \s_1^*\oplus\cdots \oplus \s_p^*.$ Since $\G$ is semi-simple, then every derivation is inner. 
Thus in particular, the derivation $\phi_{11}$ obtained in Subsection \ref{subsect:derivations-cotangent} is of the form
 $\phi_{11} = ad_{x_0}$,
for some  $x_0$ in $\G$. The semi-simplicity of $\G$ also implies that the $1$-cocycle $\phi_{12}$
obtained in Subsection \ref{subsect:derivations-cotangent} is a coboundary. That is, there exists an element $f_0$ of $\G^*$ such that 
$\phi_{12}(x) = - ad^*_xf_0,$
for all $x$ in $\G$. Here is a direct corollary of Lemma \ref{lem:intertwining}.
\bpro\label{prop:psi-semi-simple} 
If $\G$ is a semi-simple Lie algebra, then every linear map $\psi:\G^*\to\G$ which is equivariant with respect 
to the adjoint and co-adjoint actions of $\G$ and satisfies (\ref{relation-commutation}), 
is necessarily identically equal to zero.
\epro
\begin{proof}
A Lie algebra is semi-simple if and only if it contains no non-zero proper Abelian ideal. But from Lemma \ref{lem:intertwining}, 
$Im \psi$ must be an Abelian ideal of $\G$. So $Im \psi=\{0\}$.
\end{proof}
\brmq
From Theorem \ref{thm:specialcase} and Proposition \ref{prop:psi-semi-simple}, 
the cohomology space $ H^1(\D,\D)$ is determined by the space $\mathcal J$, or equivalently, by the space of
adjoint-invariant bilinear forms on $\G.$
\ermq
\bcor
If $G$ is a semi-simple Lie group with Lie algebra $\G$, then the space of bi-invariant bilinear forms on $G$ 
is of dimension $\dim H^1(\D,\D).$
\ecor
\bpro\label{guimbita2}
Suppose  $\G$ is a simple Lie algebra over $\K=\C$. Then, \\
(a) every linear map $j:\G\to\G$ in $\mathcal J$  is of the form $j(x)=\lambda x$, for some  $\lambda$ in $\C$; \\
(b) every element $\xi$ of $\mathcal E$ is of the form
$\xi = ad^*_{x_0} + \lambda id_{\G^*},$
for some  $x_0$ in $\G$ and   $\lambda$ in $\C$.
\epro
\begin{proof}
Part (a) is obtained from relation (\ref{aby}) and Schur's lemma.
From Propositions \ref{decomposition-xi} and \ref{decomposition-xi'}, for every $\xi$ in
$\mathcal E,$ there exist $\alpha$ in $\der(\G)$ and $j$ in $\mathcal J$  such that
$\xi^t=-\alpha + j$.
As $\G$ is simple and from part (a), there exist $x_0$ in $\G$ and $\lambda$ in $\C$ such that $\xi^t=-ad_{x_0}+\lambda id_\G.$
\end{proof}
\bpro
Let $G$ be a simple Lie  group with Lie algebra $\G$ over $\K=\C$.
Let $\D:=\G\ltimes \G^*$ be the Lie algebra of the cotangent bundle $T^*G$ of $G$. Then, the first
cohomology space of $\D$ with coefficients in $\D$ is
$ H^1(\D,\D) \stackrel{\sim}{=} \C $.
\epro
\begin{proof}
Indeed, a derivation $\phi : \D \to \D$ can be written, for every element
$(x,f)$ of  $\D$, as
$\phi(x,f) = ([x_0,x]\, , \, ad^*_{x_0}f - ad^*_xf_0 + \lambda f)$,  
where $ x_0$ and $f_0$ are fixed elements in $\G$ and  $\G^*$ respectively.
The inner derivations are those with $\lambda=0$.
It follows that the first  cohomology space of $\D$ with values in $\D$ is given by
$H^1(\D,\D) = \{\phi : \D \to \D : \phi(x,f)=(0,\lambda f), \lambda \in \C\}$
$ = \{\lambda(0, id_{\G^*}), \lambda \in \C \} $ $= \C id_{\G^*}$. 
 \end{proof}
As a direct consequence, we get the
\bcor 
If  $\G$ is a semi-simple Lie algebra over $\C,$ then $dim H^1(\D,\D) = p,$ where $p$ is the number of simple components 
of $\G$. 
\ecor
Consider  a semi-simple Lie algebra $\G$ over $\C$ and set $\G=\s_1 \oplus \cdots
\oplus \s_p,\; p \in \N^*$, where $\s_i,\; i=1,\ldots,p$, are simple Lie algebras. From Lemma \ref{lem:decomp-semi-simpl}, 
$\xi'$ preserves each $\s_i.$
Thus from Proposition \ref{guimbita2}, the restriction $\xi'_i$ of $\xi'$ to each $\s_i$ equals
$\xi'_i=ad_{x_{0_i}} + \lambda_i id_{\s_i}$, for some $x_{0_i}$ in $\s_i$ and
some $\lambda_i\in\C$.
 Hence, $\xi'=ad_{x_0} \oplus_{i=1}^p \lambda_i id_{\s_i}$,
where $x_0=x_{0_1}+ \cdots + x_{0_p} \in \s_1\oplus \cdots \oplus \s_p$
and $\oplus_{i=1}^p \lambda_i id_{\s_i}$ acts on an element $(x_{1}+ \cdots + x_{p})$ of $\s_1\oplus \cdots \oplus \s_p$
as follows: $(\oplus_{i=1}^p \lambda_i id_{\s_i})(x_{1}+ \cdots + x_{p})=
\lambda_1 x_{1}+ \cdots + \lambda_p x_{p}$.
In particular, we have proved
\bcor
Consider the decomposition of a semi-simple Lie algebra $\G$ over $\C$ into a sum
$\G=\s_1 \oplus \cdots \oplus \s_p,$  of simple Lie algebras
$\s_i$, $i=1,\ldots,p\in \N^*$. If a linear map $j:\G\to\G$ satisfies $j[x,y]=[jx,y],$
then there exist  $\lambda_1, \ldots,\lambda_p$ in $\C$ such that
$j=\oplus_{i=1}^p\lambda_iid_{\s_i}.$ More precisely $j(x_1+\cdots + x_p)=\lambda_1x_1+\cdots + \lambda_px_p,$
if $x_i$ is in $\s_i, ~ i=1,\ldots,p.$
\ecor
Now, we already know from Proposition \ref{prop:psi-semi-simple}, that $\psi$ vanishes identically.
So, a $1$-cocycle $\phi$ of $\D$ is given by: for every $x$ in $\G$ and every $f:=f_1 +\cdots +f_p$ 
in $\s^*_1\oplus \cdots \oplus \s^*_p=\G^*$, 
\beq\label{eq:cocycle-semi-simple}
\phi(x,f)=([x_0,x], ad^*_{x_0}f-ad^*_xf_0 + \sum_{i=1}^p\lambda_if_i),
\eeq
where $x_0$ is in $\G$, $f_0$ is in $\G^*$ and $\lambda_i$,
$i=1,\ldots,p$, are in $\C$. We then have,
\bpro
Let $\G$ be a semi-simple Lie algebra $\G$ over $\C$ and  $\D:= T^*\G$ as above. Then, the first
 cohomology space of $\D$ with respect to the adjoint action, is given by
$ H^1(\D,\D) \stackrel{\sim}{=} \C^p$, where $p$ is the number of the simple
components of $\G$.
\epro
Let $G$ be a semi-simple Lie group with Lie algebra $\G=\s_1 \oplus \cdots \oplus \s_p,$  where $\s_i$ are simple subalgebras 
(ideals, in fact) over $\mathbb C$. Now, 
make the Abelian Lie algebra $\mathbb C^p$ naturally act on $\G^*$ by means of linear maps $\rho(\lambda)$ as follows. 
If $\lambda=(\lambda_1,\ldots, \lambda_p)$ is in $\mathbb C^p$  
and  $f=f_1+\cdots+f_p$  belongs to $\G^*$, where $f_i\in \s^*_i,$ we set $\rho (\lambda)f:= \lambda_1f_1+\cdots+\lambda_pf_p. $ 
This lifts up to an action $\tilde \rho$ of $\mathbb C^p$ on $T^*\G$ by $\tilde \rho (\lambda) (x,f):= (0,\rho(\lambda) f),$ 
for all $\lambda\in \mathbb C^p$ and $ (x,f)\in  T^*\G.$ Now the latter integrates to an action $\mathcal L$ of the 
Abelian Lie group $(\mathbb C^p,+)$ on $T^*G$ by, if  $\lambda=(\lambda_1,\ldots, \lambda_p)\in \mathbb C^p$  
and  $(\sigma,f)\in T^*G$ as above, then $\mathcal L(\lambda)(\sigma, f):= (\sigma, e^{\lambda_1}f_1+\cdots+e^{\lambda_p}f_p).$ 
We will simply write $\lambda\cdot f$ for $e^{\lambda_1}f_1+\cdots+e^{\lambda_p}f_p.$
\bthm 
Let $G$ be a semi-simple Lie group and $\mathcal G$ its Lie algebra over $\mathbb C$.
Then the group $Aut(T^*\mathcal G)$ of automorphisms of the Lie algebra $T^*\mathcal G$ is, at least locally, isomorphic 
to the semi-direct product $\mathbb C^p\ltimes_{\mathcal L} T^*G,$ where $p$ is the number of simple component of $\mathcal G$ and
$\mathcal L (\lambda) (\sigma,f)= (\sigma,\lambda\cdot f),$ for all $\lambda\in \mathbb C^p$ and
 $ (\sigma,f)\in T^*G.$ 
\ethm 
\begin{proof}
 The Lie bracket on $\mathbb C^p\ltimes_{\tilde \rho} T^*\G$ reads: if $\lambda=(\lambda_1,\ldots,\lambda_p)$, $t=(t_1,\ldots,t_p)$ are in 
 $\C^p$ and $(x,f)$, $(y,g)$ belong to $T^*\G$ as above, then
 \begin{equation} 
 [(\lambda;x,f),(t;y,g)] 
                             = \left(0; [x,y] \,,\, ad_x^*g-ad_y^*f + \sum_{i=1}^p\lambda_ig_i -\sum_{k=1}^pt_kf_k\right).
 \end{equation}
 Now consider the linear map $\Gamma : \mathbb C^p \ltimes_{\tilde \rho}T^*\mathcal G \to \hbox{\rm der}(T^*\mathcal G)$, 
 $(\lambda;x_0,f_0) \mapsto \Phi_{\lambda,x_0,f_0}$, where $\Phi_{\lambda,x_0,f_0}$ is defined by
 (\ref{eq:cocycle-semi-simple}). Let $(\lambda;x_0,f_0)$ and $(t;x_0',f_0')$ be two elements of 
 $\mathbb C^p\ltimes_{\tilde \rho}T^*\mathcal G$
  such that $\Gamma(\lambda;x_0,f_0)=\Gamma(t;x_0',f_0')$. Then, for any element $(x,f)$ of $T^*\G$, 
  we have $\Phi_{\lambda,x_0,f_0}(x,f)=\Phi_{t,x_0',f_0'}(x,f)$. 
  The latter implies that, for every element $(x,f)$ of $T^*\mathcal G$, 
 $ \big([x_0,x], ad_{x_0}^*f-ad_x^*f_0 + \sum_{i=1}^p\lambda_if_i \big) =\big([x_0',x], ad_{x_0'}^*f-ad_x^*f_0' 
 + \sum_{i=1}^pt_if_i \big)$. 
 We then have, on one hand, $[x_0,x]=[x_0',x]$, for any $x$ in $\mathcal G$. 
 Since the centre of $\G$ is trivial, then $x_0'=x_0$. On the other 
 hand,  taking account the fact that $x_0=x_0'$, we have, 
 $-ad_x^*f_0 + \sum_{i=1}^p\lambda_if_i  =-ad_x^*f_0' + \sum_{i=1}^pt_if_i$, for all 
 $x\in \G$ and all $f\in \G^*$. If, in particular we take $f=0$ in the latter equality, we obtain $ad_x^*f_0 =ad_x^*f_0'$, 
 for all  $x\in \G$.  Again the triviality  of the centre of $\G$ implies that $f_0=f_0'$. 
 It is now easy to see that $\lambda_i=t_i$, 
 for all $i=1,\ldots,p$. We have then 
 prove that  $(\lambda;x_0,f_0)=(t;x_0',f_0')$ and hence, $\Gamma$ is injective. 
 Suppose, for the surjectivity, that $\Phi$ is a derivation
 of $T^*\G$, then  $\Phi$ is given by (\ref{eq:cocycle-semi-simple}), 
 for some $x_0$ in $\G$, $f_0$ in $\G^*$ and $\lambda=(\lambda_1,\ldots,\lambda_p) \in \C^p$. 
 Hence,  $\Gamma(\lambda;x_0,f_0)=\Phi$. So, $\Gamma$ is an isomorphism between the vector spaces 
 $\mathbb C^p\ltimes_{\tilde \rho} T^*\G$ and $\hbox{\rm der}(T^*\G)$.
 
 Let us now show that $\Gamma$ is  compatible with the brackets of the two Lie algebras. 
 Consider two elements $(\lambda;x_0,f_0)$ and $(t;x_1,f_1)$ 
 of $\mathbb C^p\ltimes_{\tilde \rho}T^*\G$. We have, on one hand: 
 \beq
 \Gamma\!\big([(\lambda;x_0;f_0),(t;x_1;f_1)]\big)\!\! =\!\! \Gamma\!\!\left(\!0;[x_0,x_1], ad_{x_0}^*f_1\!\!-\!\!ad_{x_1}^*f_0 
 \!\!+\!\! \sum_{i=1}^p\lambda_if_{1i} \!\!-\!\!\sum_{k=1}^pt_kf_{0k}\!\!\right) \!\!=\!\! \Phi_{0,[x_0,x_1],F}, \nonumber
\eeq  
 where for simplicity we have set $F=ad_{x_0}^*f_1-ad_{x_1}^*f_0 +\sum_{i=1}^p\lambda_if_{1i}-\sum_{k=1}^pt_kf_{0k}$.
 Now, given an element $(x,f)$ of $T^*\G$, we have
 \beqn
 \Phi_{0,[x_0,x_1],F}(x,f)\!\!\!\!&=&\!\!\!\!\! \left(\!\!\Big[[x_0,x_1],x\Big], ad_{[x_0,x_1]}^*f\!-\!ad_x^*\big(ad_{x_0}^*f_1\!-\!ad_{x_1}^*f_0 \!+\!  
                               \sum_{i=1}^p\lambda_if_{1i}\!-\!\sum_{k=1}^pt_kf_{0k}\big)\!\!\right)\cr 
                           &=& \Big(\Big[[x_0,x_1],x\Big], ad_{[x_0,x_1]}^*f - ad_x^*\circ ad_{x_0}^*f_1 +ad_x^*\circ ad_{x_1}^*f_0\cr 
                           & & -\sum_{k=1}^p\lambda_kad_{x}^*f_{1k}+ \sum_{k=1}^pt_kad_{x}^*f_{0k} \Big)    \label{eq:isomorphism-gamma1}
 \eeqn 
On the other hand, $[\Gamma(\lambda;x_0,f_0),\Gamma(t;x_1,f_1)]\!=\![\Phi_{\lambda,x_0,f_0},\Phi_{t,x_1,f_1}]$. For all $(x,f)$ in $T^*\G$, 
\beqn
[\Phi_{\lambda,x_0,f_0},\Phi_{t,x_1,f_1}](x,f)\!\!\!\!&=&\!\!\!\! \Phi_{\lambda,x_0,f_0}\Big([x_1,x], ad_{x_1}^*f-ad_x^*f_1 + \sum_{i=1}^pt_if_i\Big) \cr 
                                         & & -\Phi_{t,x_1,f_1}\Big([x_0,x], ad_{x_0}^*f-ad_x^*f_0 + \sum_{i=1}^p\lambda_if_i \Big) \cr 
                                         &=& \left(\Big[x_0,[x_1,x]\Big], ad_{x_0}^*\Big(ad_{x_1}^*f-ad_x^*f_1 + \sum_{i=1}^pt_if_i\Big)-
                                             ad_{[x_1,x]}^*f_0 \right. \cr 
                                         & &   \left. + \sum_{k=1}^p\lambda_k\Big(ad_{x_1}^*f-ad_x^*f_1 + \sum_{i=1}^pt_if_i\Big)_k\right) \cr 
                                          & & - \left(\Big[x_1,[x_0,x]\Big], ad_{x_1}^*\Big(ad_{x_0}^*f-ad_x^*f_0 + \sum_{i=1}^p\lambda_if_i\Big)-
                                             ad_{[x_0,x]}^*f_1 \right. \cr 
                                         & &   \left. + \sum_{k=1}^pt_k\Big(ad_{x_0}^*f-ad_x^*f_0 + \sum_{i=1}^p\lambda_if_i\Big)_k\right) \cr 
                                         &=& \Big(\Big[[x_0,x_1],x\Big], ad_{[x_0,x_1]}^*f - ad_x^*\circ ad_{x_0}^*f + ad_x^*\circ ad_{x_1}^*f_0 \cr   
                                         & & -\sum_{k=1}^p\lambda_kad_{x_k}^*f_1 + \sum_{k=1}^pt_kad_{x_k}^*f_0 \Big). \label{eq:isomorphism-gamma2}
\eeqn 
From (\ref{eq:isomorphism-gamma1}) and (\ref{eq:isomorphism-gamma2}), we get that 
 $\Gamma([\lambda;x_0,f_0],[t;x_1,f_1])=[\Gamma(\lambda;x_0,f_0),\Gamma(t;x_1,f_1)]$. That is $\Gamma$ is actually an isomorphism
between the Lie algebras $\mathbb C^p\ltimes_{\tilde \rho} T^*\G$ and $\hbox{\rm der}(T^*\G)$. We have then prove that 
$\hbox{\rm Aut}(T^*\G)$ is locally isomorphic to $\mathbb C^p\ltimes_\mathcal L T^*G$.
\end{proof}

%

\subsection{Case of compact Lie algebras}\label{chap:compact}
It is  known that a compact Lie algebra $\G$ decomposes as the direct sum
$\G=[\G,\G]\oplus Z(\G)$ of its derived ideal $[\G,\G]$ and its centre $Z(\G)$, with $[\G,\G]$ semi-simple and compact. 
This yields a decomposition $\G^*=[\G,\G]^*\oplus Z(\G)^* $  of $\G^*$ into a direct sum of the dual spaces 
$[\G,\G]^*$, $Z(\G)^*$ of $[\G,\G]$ and $Z(\G)$ respectively.
 On the other hand, $[\G,\G]$ also decomposes into a direct sum $[\G,\G]=\s_1\oplus\cdots \oplus\s_p$ of simple ideals $\s_i.$
From Theorem \ref{derivationschar}, a derivation $\phi$ of $\D := T^*\G$ has the  form 
$\phi (x,f) =\Big(\alpha(x) + \psi(f),\beta(x) + \xi(f)\Big),$
with conditions listed in  Theorem \ref{derivationschar}.
From Lemma \ref{lem:intertwining}, $Im \psi$ is an Abelian ideal of $\G$;
thus $Im \psi \subset Z(\G)$.
As a consequence, we have $\psi (ad_x^*f)=[x,\psi(f)]=0$, for any $x$ of $\G$
and  $f$ of $\G^*$.
\blem \label{lem:semi-simple-psi}Let $(\tilde \G,\mu)$ be an orthogonal Lie algebra  satisfying $\tilde\G= [\tilde\G,\tilde\G].$ 
Then, every $g$ in $\tilde\G^*$ is a finite sum of elements of the form $g_i=ad^*_{\bar x_i}\bar g_i,$ 
for some $\bar x_i$ in $\tilde \G,$ $\bar g_i$ in $ \tilde \G^*.$
\elem
 \begin{proof} Indeed, consider an isomorphism $\theta: \tilde\G \to \tilde\G^*$ of $\tilde\G$-modules. 
 For every $g$ in $\tilde\G^*,$ there exists $x_g$ in $\tilde\G$ such that $g=\theta (x_g).$ 
 But as $\tilde\G= [\tilde\G,\tilde\G],$ we have $ x_g=[x_1,y_1]+\cdots +[x_s,y_s]$ for some $x_i,y_i$ in $\tilde\G.$
Thus $g=\theta ([x_1,y_1])+\cdots + \theta([x_s,y_s])$ $=ad^*_{x_1}\theta(y_1)+\cdots + ad^*_{x_s}\theta(y_s)$ 
$=$ $ad^*_{\bar x_1}\bar g_1$ $+\cdots +ad^*_{\bar x_s}\bar g_s$ where $\bar x_i=x_i$ and $ \bar g_i= \theta(y_i).$ 
\end{proof}
 A semi-simple Lie algebra being orthogonal (with, e.g.  its Killing form as $\mu$), then
each $\psi$ in $\Psi$ vanishes on $[\G,\G]^*.$  This is due to Lemma \ref{lem:semi-simple-psi} 
 and the equality  $\psi (ad_x^*f)=0$, for all $x$ in $\G$, $f$ in $\G^*.$

 Of course, the converse is true. Every linear map 
$\psi:\G^*\to\G$ with $Im \psi \subset Z(\G)$ and $\psi([\G,\G]^*)=0$, is in $\Psi.$ Hence we can make the following identification.
 \blem Let $\G$ be a compact Lie algebra, with centre $Z(\G).$ Then $\Psi$ is isomorphic to the space $L(Z(\G)^*,Z(\G))$ 
 of linear maps $Z(\G)^*\to Z(\G).$
\elem
 The restriction of the cocycle $\beta$ to the semi-simple ideal $[\G,\G]$ is a coboundary,
that is, there exits an element $f_0$ in $\G^*$ such that for any $x_1$ in $[\G,\G]$,
$\beta(x_1) = -ad_{x_1}^*f_0.$ Now for $x_2$ in $Z(\G)$, one has
$ 0= \beta[x_2,y]=-ad^*_y\beta(x_2)$, for all $y$ of  $\G$.
In other words, $\beta(x_2) ([y,z]) = 0$, for all $y,z$ in $\G$. That is, $\beta(x_2)$
vanishes on $[\G,\G]$ for every $x_2 \in Z(\G)$. Hence, we write
$ \beta(x) = -ad_{x_1}^*f_0 + \eta(x_2),$  for all $x:=x_1+x_2$ in $[\G,\G] \oplus Z(\G),$
where $\eta : Z(\G) \to Z(\G)^*$ is a linear map. This simply means the following well known result.
\blem
 Let $\G$ be a compact Lie algebra, with centre $Z(\G).$ Then the first space $H^1(\G,\G^*)$ of the cohomology 
 associated with the co-adjoint action of $\G,$ is isomorphic to the space $L(Z(\G),Z(\G)^*)$.
\elem
It is known that if $\G$ is a compact Lie algebra with centre $Z(\G)$, then
$H^1(\G,\G)\!\!\stackrel{\sim}{=}\!\! \hbox{\rm End}(Z(\G))$. 
We have also already seen that $\xi$ is such that $\xi^t=-\alpha + j$, where $\alpha$
is a derivation of $\G$ and $j$ is an endomorphism of $\G$ satisfying
$j([x,y])=[j(x),y]=[x,j(y)]$. Both $\alpha$ and $j$ preserve each of
$[\G,\G]$ and $Z(\G).$ Thus we can write
$\alpha = ad_{{x_0}_1} \oplus \varphi,$  for some  ${x_0}_1 \in [\G,\G],$
where $\varphi$ is in $\hbox{\rm End}(Z(\G))$. Here $\alpha$ acts on an element $x=x_1+x_2$,
where $x_1$ is in $[\G,\G]$, $x_2$ belongs to $Z(\G),$ as follows:
$\alpha(x)  = (ad_{{x_0}_1} \oplus \varphi)(x_1+x_2):= ad_{{x_0}_1}x_1 + \varphi(x_2).$

Now suppose, for the rest of this section, that $\G$ is a compact Lie algebra over $\C$.
 We write $j=\oplus_{i=1}^p \lambda_i id_{[\G,\G]_i} \oplus \rho$, where
$\rho$ is in $\hbox{\rm End}(Z(\G))$, $\lambda_i\in\C$ and $j$ acts on $x:=x_1 + x_2$
 as follows: if $x_1:=x_{11} + x_{12} + \cdots + x_{1p}$ is in $[\G,\G]$, $x_2$ is in $Z(\G)$ and $ x_{1i}$ belongs to $\s_i$, then
 $j(x) = \Big(\bigoplus_{i=1}^p \lambda_i id_{[\G,\G]_i} \oplus \rho\Big)
(x_{11} + x_{12} + \cdots + x_{1p}+x_2) = \sum_{i=1}^p \lambda_i x_{1i} + \rho(x_2)$. 
Hence, we have the 
\blem Suppose $\G$ is a compact Lie algebra  over $\C$ with centre $Z(\G).$\\  ~
Then $\mathcal J\stackrel{\sim}{=} \C^p\oplus  \hbox{\rm End}(Z(\G)),$ where $p$ is the number of simple components of $[\G,\G]$.
\elem
\noindent
The expression of the linear map $\xi$ now reads
$
\xi = \big[ad_{x_{01}}^*  + (\oplus_{i=1}^p \lambda_iid_{[\G,\G]_i^*})\big] \oplus \delta,
$
where $\delta^t(x_2) = \rho(x_2) + \varphi(x_2)$, for all $x_2$ in $Z(\G)$,  ${x_0}_1$ is in $[\G,\G]$ and  the $\lambda_i$'s are in  $\C$.
By identifying $\hbox{\rm End}(Z(\G)),$ $L(Z(\G)^*,Z(\G))$ and $L(Z(\G),Z(\G)^*)$ to $\hbox{\rm End}(\C^k),$  where $k=\dim(Z(\G)), $ we get
 $H^1(\D,\D) = (\hbox{\rm End}(\C^k))^4 \oplus  \C^p.$ Now, if $\G=\s_1\oplus\cdots\oplus\s_p\oplus Z(\G)$
and $\G^*=\s_1^*\oplus\cdots\oplus\s_p^*\oplus Z(\G)^*$, then according to what is said above, a derivation 
$\phi$ of $\D=T^*\G$ is given as follows: if $x=x_1+x_2$ is in $\G=[\G,\G]\oplus Z(\G)$ and   
 $f=f_1+f_2=f_{11}+f_{12}+\cdots+f_{1p}+f_2$  is in $\G^*=[\G,\G]^*\oplus Z(\G)^*$, then 
 \beq\label{eq:derivation-compact} 
 \phi(x,f) = \Big([x_{01},x_1] + \varphi(x_2) + \gamma(f_2)\,,\,ad_{x_{01}}^*f_1 -ad_{x_1}^*f_0 +\sum_{i=1}^p\lambda_if_{1i}+\eta(x_2)
 +\delta(f_2) \Big),
 \eeq 
 where  $x_{01}$ and $f_0$ are some elements of $[\G,\G]$ and $\G^*$ respectively, $\varphi$ and the transpose $\delta^t$ of $\delta$ 
 are in $\hbox{\rm End}(Z(\G))$,  $\gamma : Z(\G)^* \to Z(\G)$ and  $\eta:Z(\G) \to Z(\G)^*$ are linear maps, $\lambda_i \in \C$ 
 for all $i=1,\ldots,p$.
 Now set $\dim Z(\G)=k$ and let the Lie algebra $\mathbb C^p\times (\mathfrak{gl}(\mathbb C^k))^4$ act on $T^*\G$ as follows: 
\beq \label{eq:action-Cp-GLq4-on-cotangent}
\tilde\rho(\lambda,A)(x,f)=\left(\varphi(x_2)+\gamma(f_2), \sum_{i=1}^p\lambda_if_i+\eta(x_2)+\delta(f_2)\right), 
\eeq 
for any $(x,f)$ in $T^*\G$, where $\lambda=(\lambda_1,\ldots,\lambda_p) \in \C^p$, 
$A=(\varphi,\gamma,\eta,\delta) \in (\mathfrak{gl}(\mathbb C^k))^4$. 
The action (\ref{eq:action-Cp-GLq4-on-cotangent}) of $\mathbb C^p\times (\mathfrak{gl}(\mathbb C^k))^4$ on $T^*\G$ 
integrates to an action $\tilde{\mathcal L}$ of the Lie group $\mathbb C^p\times (\hbox{\rm GL}(\mathbb C^k))^4$  on $T^*G$. The Lie 
bracket of the Lie algebra $[\mathbb C^p\times (\mathfrak{gl}(\mathbb C^k))^4]\ltimes_{\tilde\rho} T^*\G$ reads
\beqn 
[(\lambda;A;x,f),(t;B;y,g)]\!\!\!\!\!&=&\!\!\!\!\!\Big(0;[A,B]; [x,y]+\varphi(y_2)-\varphi'(x_2)+\gamma(g_2)-\gamma'(f_2),  ad_x^*g-ad_y^*f \cr 
                           & & +\sum_{i=1}^p\lambda_ig_{1i}-\sum_{i=1}^pt_if_{1i}+\eta(y_2)-\eta'(x_2) 
                           + \delta(g_2)-\delta'(f_2)\Big).
\eeqn  
Hence, if $(\lambda;A;x_0,f_0)$ belongs to $[\mathbb C^p\times (\mathfrak{gl}(\mathbb C^k))^4]\ltimes_{\tilde\rho} T^*\G$ and 
$(x,f)$ is in $T^*\G$, we have:
\beqn 
[(\lambda;A;x_0,f_0),(0;0;x,f)]\!\!\!\!\!&=&\!\!\!\!\!\Big(0;0; [x_0,x]+\varphi(x_2)+\gamma(f_2),  ad_{x_0}^*f-ad_x^*f_0 \cr 
                           & & +\sum_{i=1}^p\lambda_if_{1i}+\eta(x_2) + \delta(f_2)\Big).
\eeqn  
We have the following result.
\bthm
Let $G$ be a compact Lie group, whose Lie algebra $\G$ (over $\C$) has a $k$-dimensional centre and $[\G,\G]$  decomposes as a direct sum 
of $p$ simple ideals.  Then the Lie group $\hbox{\rm Aut}(T^*\G)$ of automorphisms of $T^*\G$ is, at least locally, isomorphic 
to the semi-direct product $[\mathbb C^p\times (\hbox{\rm GL}(\mathbb C^k))^4]\ltimes_{\tilde{\mathcal L}} T^*G$.
\ethm
\begin{proof}
As we need a local isomorphism, we will simply show that the two Lie groups 
$[\mathbb C^p\times (\hbox{\rm GL}(\mathbb C^k))^4]\ltimes_{\tilde{\mathcal L}} T^*G$ 
and $\hbox{\rm Aut}(T^*\G)$ have isomorphic Lie algebras. But the Lie algebra of $\hbox{\rm Aut}(T^*\G)$ is $\der(T^*\G)$ while, 
by construction, the one  of $[\mathbb C^p\times (\hbox{\rm GL}(\mathbb C^k))^4]\ltimes_{\tilde{\mathcal L}} T^*G$ is 
$[\mathbb C^p\times (\mathfrak{gl}(\mathbb C^k))^4]\ltimes_{\tilde\rho} T^*\G$.
%
Now consider the  linear map $\Delta:[\mathbb C^p\times (\mathfrak{gl}(\mathbb C^k))^4]\ltimes T^*\G \to \der(T^*\G)$ given by 
$\Delta(\lambda;A;x_0,f_0):=\Phi_{\lambda;A;x_0,f_0}$,
where, if $A=(\varphi,\gamma,\eta,\delta) \in (\mathfrak{gl}(\mathbb C^k))^4$, $x_0=x_{01}+x_{02} \in \G$, $f_0 \in \G^*$, 
$\Phi_{\lambda;A;x_0,f_0}$  is defined by (\ref{eq:derivation-compact}). It is easy to see that $\Delta$ is 
actually an isomorphism between the Lie algebras $[\mathbb C^p\times (\mathfrak{gl}(\mathbb C^k))^4]\ltimes T^*\G$ 
and $\der(T^*\G)$ and hence the Lie groups $\hbox{\rm Aut}(T^*\G)$
and $[\mathbb C^p\times (\hbox{\rm GL}(\mathbb C^k))^4]\ltimes_{\tilde{\mathcal L}} T^*G$ are  locally isomorphic.
\end{proof}

\section{Some examples}\label{section:examples}
Here, we apply the above results to produce some examples in low dimensions. In the following, if $(e_1, \ldots, e_n)$ 
is a basis of some vector space $E$, then $(e_1^*, \ldots, e_n^*)$ will denote its dual basis. We will let $e_{ij}$ 
stand for the linear map of $E$ defined as $e_{ij}(e_k)=\delta_{jk}e_i$ where $\delta_{jk}$ is the Kronecker delta 
$\delta_{jj}=1$ and $\delta_{jk}=0$ if $j\neq k.$ Hence we can express the Lie bracket of endomorphisms of $E$ as 
linear combinations of $[e_{ij},e_{kl}]= \delta_{jk}e_{il}-\delta_{il}e_{kj}$. 
We may write the relations  (\ref{decompositionphi}) in a matrix form as $\phi(x,f)=\begin{pmatrix} \alpha & \psi \\ \beta & \xi\end{pmatrix}  
\begin{pmatrix} x \\ f\end{pmatrix}$, so that, keeping the same notations for maps and their matrices in a basis of $T^*\G$,
the matrix of a derivation $\phi$ of $T^*\G$  will be of the form $\phi=\begin{pmatrix} \alpha & \psi \\ \beta & \xi\end{pmatrix}$.
We note that for some of the examples we treat here, $\der(\mathcal D)$ admits a contact structure.

\subsection{The affine Lie algebra of the real line}\label{chap:aff(R)}
The 2-dimensional affine Lie algebra $\G=\hbox{\rm aff}(\R)$ is solvable nonnilpotent with Lie bracket $[e_1,e_2]=e_2$ in 
some basis ($e_1,e_2$). The Lie algebra $\D=T^*\G$ of the cotangent bundle of any Lie group with Lie algebra $\G$, 
has a basis ($e_1,e_2,e_3,e_4$) with Lie bracket $[e_1,e_2]=e_2$, $[e_1,e_4]=-e_4$, $[e_2,e_4]=e_3$, where $e_3:=e_1^*$ 
and $e_4:=e_2^*$. This is the semi-direct product $ \R e_1\ltimes \mathcal H_3$ of the Heisenberg Lie algebra 
$\mathcal H_3=\sspan(e_2,e_3,e_4)$ and the line $\R e_1$, where $e_1$ acts on $\mathcal H_3$ by the restriction of 
the derivation $ad_{e_1}$. 
According to \cite{diatta-manga-frobenius}, any derivation of $\mathfrak{aff}(\R^n),$~$n\ge 1,$ is an inner derivation. Thus 
$\der(\G)=\sspan(ad_{e_1},ad_{e_2})$. Hence, a derivation $\alpha$ of $\G$, is of the form
$\alpha(x)=(ax^1+bx^2)e_2$,  for each $x=x^1e_1+x^2e_2$ of $\G$, where $a,b,x^1,x^2$ are in $\R$. 
Whereas, an element $\xi$ of $\mathcal E$, related to the above $\alpha$ by (\ref{relation-xi-alpha}), 
has the following  form  
$ \xi(f)=[(b+\lambda)f^3-af^4]e_3 + \lambda f^4e_4$, for $f=f^3e_3+f^4e_4$,
where $\lambda$ is in $\mathbb R$. Now the space $\Psi$ of equivariant linear maps $\psi$ is reduced to $\{0\}$ 
while a cocycle $\beta:\G\to \G^*$ is given by 
$\beta(x)=(cx^1+dx^2)e_3-dx^1e_4$,
where $c,d$ are in $\R$. Altogether,  a derivation $\phi$ of $\D=T^*\G$,  is of the form $\phi(x,f)=(\alpha(x),\beta(x)+\xi(f))$ 
and  in the basis $(e_1,e_2,e_3,e_4),$ it has the matrix
$
\phi=\begin{pmatrix}
0  & 0 &  0  & 0 \cr  
a  & b &  0  & 0 \cr 
c  & d & b+\lambda &-a \cr 
-d & 0 &  0  & \lambda \cr 
\end{pmatrix} =a (e_{21}-e_{34})+b (e_{22}+e_{33})+c e_{31} +d (e_{32}-e_{41})+ \lambda(e_{33}+e_{44})
$. 

The Lie algebra $\der(\D)$ has a basis ($\phi_1,\phi_2,\phi_3,\phi_4,\phi_5$) where 
$\phi_1:=e_{21}-e_{34}$, $\phi_2:=e_{22}+e_{33}$, $\phi_3:=e_{31}$, $\phi_4:=-e_{41}+e_{32}$, $\phi_5:=e_{33}+e_{44}$, 
so that the Lie brackets are
$[\phi_2,\phi_1]=\phi_1$, $[\phi_2,\phi_3] = \phi_3$, $[\phi_5,\phi_3]=  \phi_3$, $[\phi_5,\phi_4]=  \phi_4$. 
As a vector space, $\der(\D)$ decomposes as $\der(\D)= \G_0\oplus \G_1$, where $\G_0:=span(\phi_1,\phi_2,\phi_5)$ 
and $\G_1:=span(\phi_3,\phi_4)$. We see that $[\G_0,\G_1]= \G_1$, $[\G_0,\G_0]=\mathbb R\phi_1\subset \G_0$ and $[\G_1,\G_1]=\{0\}\subset \G_0$. 
Hence,  $\der(\D)$ is a supersymmetric Lie algebra. In fact, here, as $\Psi=\{0\}$,  then $\der(\D)$  is a Lie superalgebra (Proposition \ref{superalgebras}).
We can also rewrite  $\der(\D)$ as the semi-direct product $ \R^2\ltimes\R^3 $ of the abelian Lie algebras $\R^3=\sspan_{\R}(\phi_1,\phi_3,\phi_4)$ 
and $\R^2=\sspan_{\R}(\phi_2,\phi_5)$  so that  it is easily identified with the Lie algebra number 18, for $p=q=1,$ 
in Section 5.2 of \cite{diatta-contact}. Hence  $\der(\D)$  has a contact structure.

\subsection{The Lie algebra of the group $\hbox{\rm SO}(3)$ of rotations}\label{chap:so(3)}
Consider the Lie algebra $\G:=\mathfrak{so}(3)=\sspan(e_1,e_2,e_3)$ with $[e_1,e_2]=-e_3$, $[e_1,e_3]=e_2$, $[e_2,e_3]=-e_1$.
Thus $\D=\sspan(e_1,e_2,e_3,e_4,e_5,e_6)$ has the Lie bracket $[e_1,e_5]=-e_6$, $[e_1,e_6]=e_5$, $[e_2,e_4]=e_6$, $[e_2,e_6]=-e_4$, 
$[e_3,e_4]=-e_5$, $[e_3,e_5]=e_4$, where $e_{3+i}=e_i^*$, $i=1,2,3$. 
Since $\mathfrak{so}(3)$  is simple, we apply the results of Section \ref{chap:semi-simple}. So $\Psi=\{0\}.$ 
As  any derivation $\alpha$ is inner, that is of the form $\alpha=ad_{x_0}$, for some $x_0\in  \mathfrak{so}(3),$  thus,  
$\alpha$ and an element $\xi=ad_{x_0}^*+\lambda id_{\mathfrak{so}(3)^*}$ of $\mathcal E$ corresponding 
to $\alpha$ are given by 
$\alpha(e_1)=-ae_2-be_3$, $\alpha(e_2)=ae_1-ce_3$, $\alpha(e_3)=be_1+ce_2$ and $\xi(e_4)=\lambda e_4-ae_5-be_6$, 
$\xi(e_5)=ae_4+\lambda e_5-ce_6$, $\xi(e_6)=be_4+ce_5+\lambda e_6$,
%
for some $a,b,c,\lambda$ in $\R$.  Moreover, as a cocycle $\beta$ in $\mathcal Q$ is a coboundary, it is given by  
$\beta(x)=-ad_{x}^*f_0$, for some $f_0$ in $\mathfrak{so}(3)^*$ and any $x$ in $\mathfrak{so}(3)$. Hence, 
$\beta(e_1)=-de_5-ee_6$, $\beta(e_2)=de_4-ie_6$, $\beta(e_3)=ee_4+ie_5$,
%
for some $d,e,i$ in $\R$. Hence, in  the basis $(e_1,e_2,e_3,e_4,e_5,e_6)$, the matrix of a derivation $\phi$ of $T^*\mathfrak{so}(3)$ has the form 
$
\left(
\begin{array}{rrrrcc}
0 & a & b & 0 & 0 & 0  \cr 
-a& 0 & c & 0 & 0 & 0  \cr 
-b&-c  & 0 &0 & 0 & 0 \cr 
0 & d & e & \lambda & a & b  \cr 
-d& 0 & i&-a &\lambda & c  \cr
-e & -i & 0&-b& -c  & \lambda
\end{array}\right)
= a\phi_1+b\phi_2+c\phi_3+d\phi_4+e\phi_5+\lambda\phi_6+i\phi_7$, 
where $a,b,c,d,e,\lambda,i$ are in $\mathbb R$.
 Hence, we have $\der(\D)=\sspan(\phi_1,\phi_2,\phi_3,\phi_4,\phi_5,\phi_6,\phi_7)$, where
$\phi_1:=-e_{21}+e_{12}-e_{54}+e_{45}$, 
$\phi_2:=-e_{31}+e_{13}-e_{64}+e_{46}$,
$\phi_3:=-e_{32}+e_{23}-e_{65}+e_{56}$,
$\phi_4:=-e_{51}+e_{42}$,
$\phi_5:=-e_{61}+e_{43}$,
$\phi_6:=e_{44}+e_{55}+e_{66}$,
$\phi_7:=-e_{62}+e_{53}$,
so that the Lie algebra	 structure of $\der(\D)$  is given by  the following Lie brackets
 $[\phi_1,\phi_2]=-\phi_3$,  $[\phi_1,\phi_3]=\phi_2$,  $[\phi_1,\phi_5]=-\phi_7$,  $[\phi_1,\phi_7]=\phi_5$, $[\phi_2,\phi_3]=-\phi_1$,   
 $[\phi_2,\phi_4]=\phi_7$,  $[\phi_2,\phi_7]=-\phi_4$, $[\phi_3,\phi_4]=-\phi_5$,  $[\phi_3,\phi_5]=\phi_4$,
 $[\phi_4,\phi_6]=-\phi_4$,  $[\phi_5,\phi_6]=-\phi_5$, $[\phi_6,\phi_7]=\phi_7$.  As $\Psi=\{0\}$, then 
$\der(\D)$ is a super Lie algebra. We can also see that, it is the Lie algebra 
 $\der(\D)=\mathfrak{so}(3)\ltimes \G_{id}$, 
 where $so(3)=\sspan(\phi_1,\phi_2,\phi_3)$ and $\G_{id}$ is the semi-direct product $\G_{id}=\R \phi_6\ltimes \R^3$ of the abelian 
 Lie algebras $\R^3=\sspan(\phi_4,\phi_5,\phi_7)$ and $\R \phi_6$ obtained by letting $\phi_6$ act as the identity map on $\R^3.$ 
 Thus $\der(\D)$ is also a contact Lie algebra, as it is the Lie algebra number 4 of Section 5.3 in \cite{diatta-contact}. As $\Psi=0,$ Proposition \ref{superalgebras} insures that  $\der(\D)$ is also a Lie superalgebra.

\subsection{The Lie algebra of the group $\hbox{\rm SL}(2)$ of special linear group}\label{chap:sl(2)}
The Lie algebra $\G:=\mathfrak{sl}(2)$ of $\hbox{\rm SL}(2)$ has a basis $(e_1,e_2,e_3)$  in which its 
Lie bracket reads $[e_1,e_2]=-2e_2$,  $[e_1,e_3]=2e_3$,  $[e_2,e_3]=-e_1$. Set  $e_1^*=:e_4$, $e_2^*=:e_5$, $e_3^*=:e_6$, the Lie bracket of
$\D:=T^*\G$ in the basis $(e_1,e_2,e_3, e_4,e_5,e_6)$ is given by
$[e_1,e_2]=-2e_2$, $[e_1,e_3]=2e_3$, $ [e_2,e_3]=-e_1$, $[e_1,e_5]=2e_5$,  $[e_1,e_6]=-2e_6$, $[e_2,e_4]=e_6$,  $[e_2,e_5]=-2e_4$,
$[e_3,e_4]=-e_5$, $[e_3,e_6]=2e_4$. 
Since  $\mathfrak{sl}(2)$ is simple, the results of Section \ref{chap:semi-simple} also  apply. In particular, $\Psi=\{0\}.$ 
A derivation $\alpha=ad_{x_0}$ of $\mathfrak{sl}(2)$, for some $x_0\in \mathfrak{sl}(2),$  and an 
element $\xi=ad_{x_0}^*+\lambda id_{\mathfrak{sl}(2)^*}$ of $\mathcal E$ corresponding 
to $\alpha,$ are given by $\alpha(e_1)=-2de_2+2ce_3$, $\alpha(e_2)=-ce_1-ae_2$, $\alpha(e_3)=de_1+ae_3$ and 
$\xi(e_4)=\lambda e_4+ce_5-de_6$, $\xi(e_5)=2de_4+(\lambda+a)e_5$, $\xi(e_6)=-2ce_4+(\lambda-a)e_6$,
for some $a,\lambda,c,d$ in $\R$.  Moreover, a cocycle $\beta \in \mathcal Q$ is  given by  
$
\beta(x)=-ad_{x}^*f_0, 
$ for some $f_0$ in $\mathfrak{sl}(2)^*$ and any $x$ in $\mathfrak{sl}(2)$. We have  
$\beta(e_1)=-he_5-ge_6$, $\beta(e_2)=he_4+ee_6$ and $\beta(e_3)=ge_4-ee_5$,
 for some $e,g,h$ in $\R$. Hence, the matrix of a derivation $\phi$ of $T^*\mathfrak{sl}(2)$   in the basis $(e_1,e_2,e_3,e_4,e_5,e_6),$ is
\beq \phi=
\left(
\begin{array}{rrrrcc}
0 & -c & d & 0 & 0 & 0  \cr 
-2d& -a & 0& 0 & 0 & 0  \cr 
2c & 0  & a &0 & 0 & 0 \cr 
0 & h & g & \lambda & 2d &-2c  \cr 
-h& 0 & -e&c &a+\lambda & 0  \cr
- g & e & 0&-d& 0  &\lambda -a
\end{array}\right)\nonumber
\eeq 
$= a\phi_1+\lambda\phi_2+c\phi_3+d\phi_4+e\phi_5+h\phi_6+g\phi_7,$
 where $a,\lambda,c,d,e,h,g$ are real numbers and 
$\phi_1:=-e_{22}+e_{33}+e_{55}-e_{66},$ $\phi_2:=e_{44}+e_{55}+e_{66},$
$\phi_3:=-e_{12}+2e_{31}-2e_{46}+e_{54},$ $\phi_4:=e_{13}-2e_{21}+2e_{45}-e_{64},$ $\phi_5:=-e_{53}+e_{62},$ 
$\phi_6:=e_{42}-e_{51},$ $\phi_7:=e_{43}-e_{61}.$ 
Hence the Lie algebra $\der(\D)$ is $7$-dimensional. In the basis $(\phi_1, \ldots, \phi_7)$, its Lie bracket reads
$[\phi_1,\phi_3]=\phi_3$,
$[\phi_1,\phi_4] = -\phi_4$,
$[\phi_3,\phi_4]= 2\phi_1$,
$[\phi_1,\phi_6]=\phi_6$,
$[\phi_1,\phi_7] = - \phi_7$,
$[\phi_2,\phi_5]=\phi_5$,
$[\phi_2,\phi_6]=\phi_6$,
$[\phi_2,\phi_7]=\phi_7$,
$[\phi_3,\phi_5]=-2\phi_6$,
$[\phi_3,\phi_7]= - \phi_5$,
$[\phi_4,\phi_5]= - 2\phi_7$,
$[\phi_4,\phi_6] = - \phi_5$. One realizes that this is the Lie algebra $\der(\D))=\mathfrak{sl}(2,\R)\ltimes \G_{id}$, 
where $\mathfrak{sl}(2,\R)=\sspan(\phi_1,\phi_3,\phi_4)$ and as above, $\G_{id}$ is the semi-direct product $\G_{id}=\R \phi_2\ltimes \R^3$ 
of the abelian Lie algebras $\R^3=\sspan(\phi_5,\phi_6,\phi_7)$ and $\R \phi_2$ obtained by letting $\phi_2$ act as the identity 
map on $\R^3.$ Again,  $\der(\D)$ is a Lie superalgebra and also a contact Lie algebra, with e.g. $\eta:=s\phi_1^*+t\phi_5^*$ as 
a contact form, $s,t\in\R -\{0\}$.

\subsection{An example of a solvable Lie algebra in dimension 3} 
Consider $\G:=\sspan(e_1,e_2,e_3)$ with $[e_1,e_3]=-ae_1-e_2$, $[e_2,e_3] =e_1-ae_2$.  
If as above, we set $e_4:=e_1^*$, $e_5:=e_2^*$, $e_6:=e_3^*$, then in the basis $(e_1,e_2,e_3, e_4,e_5,e_6)$ 
the Lie bracket of $\D:=T^*\G$, reads $[e_1,e_3] =-ae_1-e_2$, $[e_1,e_4]= ae_6$, $[e_1,e_5] =e_6$, $[e_2,e_3]=e_1-ae_2$, 
$[e_2,e_4] = -e_6$, $ [e_2,e_5] =ae_6$, $ [e_3,e_4] =-ae_4+e_5$, $[e_3,e_5]= -e_4 - ae_5$.  
Elements $\alpha \in \der(\G)$  are of the form $\alpha(e_1)=\alpha_{11}e_1-\alpha_{12}e_2$,   
$\alpha(e_2)=\alpha_{12}e_1+\alpha_{11}e_2$, $\alpha(e_3)=\alpha_{13}e_1+\alpha_{23}e_2$, 
 while a bi-invariant endomorphisms of $\G$  are those $j$, such that $j(x)=\lambda$ $x$, for every $x\in \G$, 
 where  $\alpha_{11}$, $\alpha_{12}$, $\alpha_{13}$, $\alpha_{23}$, $\lambda$ are scalars.
Cocycles $\beta \in \mathcal Q$ are of  the form  $\beta(e_1)=-\beta_{12} e_5-\beta_{13}e_6$,  
$\beta(e_2)=\beta_{12} e_4-\beta_{23}e_6$, $\beta(e_3)=\beta_{13}e_4+\beta_{23} e_5+\beta_{33}e_6$,
 if $a=0$ and when $a\neq 0$ we have $\beta(e_1)=-\beta_{13}e_6$, $\beta(e_2)=-\beta_{23}e_6$, 
 $\beta(e_3)=\beta_{13}e_4+\beta_{23} e_5+\beta_{33}e_6$. 
On the other hand, when  $a=0$, the elements $\psi$ of $\Psi$ take the form   $\psi(e_1)=-\psi_{12}e_5$, $\psi(e_2)=\psi_{12}e_4$, 
and  $\psi(e_3)=0$,  whereas  $\Psi=\{0\}$ when $a\neq 0$. \\
(1) Case $a=0$:  a derivation of $\D$ have a matrix of the form
\beq
 \phi=\left(\begin{array}{cccccc}
\alpha_{11}&\alpha_{12} &   \alpha_{13}         &    0             &\psi_{12}   & 0      \\
-\alpha_{12}&\alpha_{11}& \alpha_{23}           & -\psi_{12}       &     0      & 0       \\
   0        &    0      &   0                   &    0             & 0          & 0  \\
   0        & \beta_{12}& \beta_{13}            &\lambda-\alpha_{11}& \alpha_{12}& 0 \\
-\beta_{12} &    0      & \beta_{23}            &      -\alpha_{12}&\lambda -\alpha_{11} & 0 \\ 
-\beta_{13} &-\beta_{23}& \beta_{33}            &      -\alpha_{13}&      -\alpha_{23} &\lambda
\end{array}\right) \nonumber
\eeq
$
=\alpha_{11}\phi_1+\alpha_{12}\phi_2+\alpha_{13}\phi_3 +\alpha_{23}\phi_4  +\lambda \phi_5   
 +\psi_{12}\phi_6 +  \beta_{12}\phi_7 + \beta_{13}\phi_8 + \beta_{23}\phi_9 +  \beta_{33}\phi_{10}, \nonumber
$ 
   with $\phi_{1}:=e_{11}+e_{22}-e_{44}-e_{55}$, 
 $\phi_{2}:=e_{12}-e_{21}+e_{45}-e_{54}$,  
 $\phi_{3}:=e_{13}-e_{64}$,
  $\phi_{4}:=e_{23}-e_{65}$,  
$\phi_{5}:=e_{44}+e_{55}+e_{66}$, 
 $\phi_{6}:=e_{15}-e_{24}$,
 $\phi_{7}:= e_{42}-e_{51}$, 
 $\phi_{8}:= e_{43}-e_{61}$,  
$\phi_{9}:=e_{53}-e_{62}$,  
$\phi_{10}:= e_{63}$.
 So that,  $\dim\der(\D)=10$ and $\der(\D)$ has Lie bracket: 
 $[\phi_{1},\phi_{3}]=\phi_{3}$, $[\phi_{1},\phi_{4}]=\phi_{4}$, $[\phi_{1},\phi_{6}]=2\phi_{6}$, 
 $[\phi_{2},\phi_{3}]=-\phi_{4}$, $[\phi_{2},\phi_{4}]=\phi_{3}$, $[\phi_{2},\phi_{8}]=-\phi_{9}$,  
 $[\phi_{2},\phi_{9}]=\phi_{8}$,  
$[\phi_{3},\phi_{7}]=\phi_{9}$, 
$[\phi_{4},\phi_{7}]=-\phi_{8}$, 
$[\phi_{5},\phi_{6}]=-\phi_{6}$, 
$[\phi_{5},\phi_{7}]=\phi_{7}$, 
$[\phi_{5},\phi_{8}]=\phi_{8}$, 
$[\phi_{5},\phi_{9}]=\phi_{9}$, 
$[\phi_{5},\phi_{10}]=\phi_{10}$, 
$[\phi_{6},\phi_{7}]=-\phi_{1}$, 
$[\phi_{6},\phi_{8}]=-\phi_{4}$, 
$[\phi_{6},\phi_{9}]=\phi_{3}$.  Moreover, $\G_0\oplus\tilde\G_1:=\sspan(\phi_{1}, \phi_{2}, \phi_{3}, \phi_{4},\phi_{5},\phi_{6})$ 
and  $\G_0\oplus\tilde\G_1':=\sspan(\phi_{1}, \phi_{2}, \phi_{3}, \phi_{4},\phi_{5},\phi_{7},\phi_{8}, \phi_{9}, \phi_{10})$ are Lie  superalgebras.
\\
(2) Case $a\neq 0$: the derivations of $\D$ are
\beq
\phi=  \left(\begin{array}{cccccc}
\alpha_{11}&\alpha_{12} &   \alpha_{13}         &    0             &   0  & 0      \\
-\alpha_{12}&\alpha_{11}& \alpha_{23}           &    0       &     0      & 0       \\
   0        &    0      &   0                   &    0             & 0          & 0  \\
   0        &    0      & \beta_{13}            &\lambda-\alpha_{11}&        \alpha_{12}& 0 \\
   0        &    0      & \beta_{23}            &      -\alpha_{12}&\lambda-\alpha_{11} & 0 \\ 
-\beta_{13} &-\beta_{23}& \beta_{33}            &      -\alpha_{13}&      -\alpha_{23} &\lambda
\end{array}\right)\nonumber
\eeq
$=\alpha_{11}\phi_1+\alpha_{12}\phi_2+\alpha_{13}\phi_3 +\alpha_{23}\phi_4 +\lambda \phi_5  
+ \beta_{13}\phi_8 + \beta_{23}\phi_9 +  \beta_{33}\phi_{10}$,  so $\dim\der(\D)=8.$  
The Lie bracket of $\der(\D$) is:
$[\phi_{1},\phi_{3}]=\phi_{3}$, 
$[\phi_{1},\phi_{4}]= 
 -[\phi_{2},\phi_{3}]=\phi_{4}$,
 $[\phi_{2},\phi_{4}]=\phi_{3}$,  
 $[\phi_{2},\phi_{8}]=-\phi_{9}$,  
 $[\phi_{2},\phi_{9}]= 
[\phi_{5},\phi_{8}]=\phi_{8}$, 
$[\phi_{5},\phi_{9}]=\phi_{9}$, 
$[\phi_{5},\phi_{10}]=\phi_{10}$.   
From Proposition \ref{superalgebras},  $\der(\D$) is a Lie  superalgebra.

\subsection{The $4$-dimensional oscillator algebra}

The $4$-dimensional oscillator Lie algebra is the space $\G= \sspan\{e_{1},e_2,e_3,{e}_4\}$ with bracket:
$[e_{1},e_3] =  e_4$, $[e_{1}, e_4] = -  e_3$, $[e_3, e_4] = e_2$. 
Set $e_5:=e^*_{1},$~$e_6:=e^*_2,$~$e_7:=e^*_3,$~$ e_8:=e^*_4.$
The Lie bracket of the  Lie algebra $\D=T^*\G$  reads 
$[e_{1},e_3] =  e_4$, $[e_{1}, e_4] = -  e_3$, $[e_3, e_4] = e_2,$ 
$[e_{1},e_7] =  e_8$, $ [e_{1}, e_8] = -  e_7$, $[e_3, e_6] =- e_8$, 
$[ e_4,e_7] =-  e_{5}$, $[e_3, e_8]= e_{5}$, $[ e_4, e_6] = e_7$.
We define an orthogonal structure $\mu$   on $\G$ by
$\mu(x,y) = x_{1}y_2 + x_2y_{1} + x_3y_3+ x_4  y_4$,
for all $x=x_{1}e_{1}+x_2e_2+x_3e_3+ x_4 e_4$ and $y=y_{1}e_{1}+y_2e_2+y_3e_3+ y_4 e_4$,   
see \cite{bromberg-medina2004}.
The  isomorphism $\theta : \G \to \G^*$ defined by $\langle \theta (x),y\rangle = \mu(x,y)$, for all $x,y$ in $\G$, also reads 
$\theta(e_{1}) = e_6$, $\theta(e_2) = e_{5}$, $\theta(e_3)=e_7$, $\theta( e_4)= e_8$.
Any derivation $\phi$ of $T^*\G$ is of the form
$
\phi(x,f) = \Big( \alpha_1(x) + j_2 \circ \theta^{-1}(f) \,,\, \theta\circ \alpha_2(x) + (j_1-\alpha_1)^t(f)\Big)
$, for every $(x,f)$ in $T^*\G$, where $\alpha_1,\alpha_2$ are in $\der(\G)$ and $j_1,j_2$ are in $\mathcal{J}$.
Derivations  $\alpha_1$ and $\alpha_2$ of $\G$ are given by
$\alpha_1(e_{1})=a_{21}e_2-a_{23}e_3-a_{24}e_4$, $\alpha_1(e_2)=2a_{33}e_2$, 
$\alpha_1(e_3)=a_{23}e_2+a_{33}e_3-a_{34} e_4$, $\alpha_1( e_4)=a_{24}e_2+a_{34}e_3+a_{33} e_4$ and 
$\alpha_2(e_{1})=b_{21}e_2-b_{23}e_3-b_{24} e_4$, $\alpha_2(e_2)=2b_{33}e_2$, 
$\alpha_2(e_3)=b_{23}e_2+b_{33}e_3-b_{34} e_4$, $\alpha_2( e_4)=b_{24}e_2+b_{34}e_3+b_{33} e_4$, where 
$a_{ij},b_{ij}\in\mathbb \R.$
Now bi-invariant tensors $j_1,j_2:\G \to \G$ are given by
$j_1(e_{1})=\lambda e_{1}+ae_2$, $j_2(e_{1})=\lambda' e_{1}+be_2$, $j_1(e)=\lambda e$ and $j_2(e)=\lambda' e$, if $e\in \{e_2,e_3,e_4\}$, 
where $\lambda$, $\lambda'$, $a$ and $b$ are real numbers.
The condition $Im(j_2) \subset Z(\G) =\mathbb R e_2$ gives  $j_2(e_{1})=b'e_2,$   $j_2(e_2)=j_2(e_3)=j_2( e_4)=0.$  
 We get $j_2\circ\theta^{-1}(e_6)=be_2$ and 
$j_2\circ\theta^{-1}(e_{5})=j_2\circ\theta^{-1}(e_7)=j_2\circ\theta^{-1}(e_8)=0$.
Last, we have $\theta \circ \alpha_2(e_{1})=b_{21}e_{5}-b_{23}e_7-b_{24} e_8$, $\theta\circ\alpha_2(e_2)=2b_{33}e_{5}$, 
$\theta\circ\alpha_2(e_3)=b_{23}e_{5}+b_{33}e_7-b_{34} e_8$, 
$\theta\circ\alpha_2(\check e_3)=b_{24}e_{5}+b_{34}e_7+b_{33} e_8.$ 
%
Thus, derivations $\phi\in\der(\D)$ are of the form\\
\beq \phi=\left(
\begin{array}{cccccccc} 
0               &   0        &    0      &     0   &          0       &    0     &       0        &       0        \cr
a_{21}          & 2a_{33}    &   a_{23}  &  a_{24} &          0       &     b    &       0        &       0        \cr
-a_{23}         &   0        &   a_{33}  &  a_{34} &          0       &    0     &       0        &       0        \cr
-a_{24}         &   0        &  -a_{34}  &  a_{33} &          0       &    0     &       0        &       0        \cr
b_{21}          & 2b_{33}    &  b_{23}   &  b_{24} & \lambda  & a-a_{21} &    a_{23}     &    a_{24}     \cr
0               &   0        &     0     &   0     &          0       & \lambda -2a_{33} &       0        &       0        \cr
- b_{23}        &   0        &   b_{33}  & b_{34}  &          0       &- a_{23}   & \lambda-a_{33} &    a_{34}     \cr
- b_{24}        &   0        & -b_{34}   & b_{33}  &          0       &-a_{24}   &   -  a_{34}     & \lambda-a_{33}
\end{array}\right)\nonumber\eeq
\\ $= 
\lambda \phi_1     +  a_{33}\phi_2   + a_{21}\phi_3+a_{23}\phi_4+a_{24}\phi_5+a_{34}\phi_6    +b\phi_7+b_{21}\phi_8 +b_{33}\phi_9+b_{23}\phi_{10}+b_{24}\phi_{11}  +b_{34}\phi_{12}+a\phi_{13},$   ~
  where     $\phi_1 =e_{55}+e_{66}+e_{77}+e_{88}$,  $\phi_2 =2e_{22}+e_{33}+e_{44}-2e_{66}-e_{77}-e_{88},$  $\phi_3 =e_{21}-e_{56},$  $\phi_4 =e_{23}-e_{31}-e_{76}+e_{57}$,  $\phi_5 =e_{24}-e_{41}-e_{86}+e_{58}$,  $\phi_6 =e_{34}-e_{43}-e_{87}+e_{78},$  $\phi_7 =e_{26},$   $\phi_8 =e_{51},$  $\phi_9 =2e_{52}+e_{73}+e_{84},$  $\phi_{10} =e_{53}-e_{71},$  $\phi_{11} =e_{54}-e_{81},$  $\phi_{12} =e_{74}-e_{83},$ $\phi_{13}=e_{56}.$
So $\der(\D)$   has dimension $13$ , with Lie bracket          $[\phi_1,\phi_7]=-\phi_7$, $[\phi_1,\phi_8]=\phi_8$, $[\phi_1,\phi_{9}]=\phi_{9}$,
$[\phi_1,\phi_{10}]=\phi_{10}$, $[\phi_1,\phi_{11}]=\phi_{11}$, $[\phi_1,\phi_{12}]=\phi_{12}$,
$[\phi_2,\phi_3]=2\phi_3$, $[\phi_2,\phi_4]=\phi_4$, $[\phi_2,\phi_5]=\phi_5$, $[\phi_2,\phi_7]=4\phi_7$, 
$[\phi_2,\phi_{9}]=-2\phi_{9}$, $[\phi_2,\phi_{10}]=-\phi_{10}$, $[\phi_2,\phi_{11}]=-\phi_{11}$, 
$[\phi_2,\phi_{12}]=-2\phi_{12}$, $[\phi_2,\phi_{13}]=2\phi_{13}$, $[\phi_3,\phi_9]=-2\phi_8$,
$[\phi_4,\phi_6]=\phi_5$, $[\phi_4,\phi_9]=-\phi_{10}$, $[\phi_4,\phi_{12}]=\phi_{11}$,
$[\phi_5,\phi_6]=-\phi_4$, $[\phi_5,\phi_{9}]=-\phi_{11}$, $[\phi_5,\phi_{12}]=-\phi_{10}$, 
$[\phi_6,\phi_{10}]=-\phi_{11}$, $[\phi_6,\phi_{11}]=\phi_{10}$, $[\phi_7,\phi_{9}]=-2\phi_{13}$. 
Let $\G_0:=\sspan(\phi_1,\phi_2,\phi_3, \phi_4,\phi_5, \phi_6, \phi_{13}),$
 $\mathcal Q:=\sspan(\phi_8,\phi_9,\phi_{10},\phi_{11}, \phi_{12}),$ $\Psi=\mathbb R\phi_7.$ Then $\G_0\oplus\mathcal Q$ and  $\G_0\oplus\Psi$ are Lie superalgebras. 

\section{Conclusion}\label{openproblems}
Given two left or right invariant structures of the same `nature' (e.g. physical, affine, symplectic, complex,
Riemannian,\ldots) on $T^*G$, one wonders whether they are
equivalent, {\it i.e.} if there exists an automorphism of $T^*G$ mapping one to the other.
By taking the values of those structures at the unit of $T^*G,$
the problem translates to finding an automorphism of Lie algebra mapping
two structures of $\D.$ The work within this paper may also be seen as a useful tool for the study of such structures.  
For more discussions on structures and problems on $T^*G$, see e.g.
\cite{alekseevsky-grabowski94}, \cite{bajo-benayadi-medina}, \cite{di-me-cybe}, \cite{drinfeld}, \cite{feix}, \cite{kronheimer}, 
\cite{marle}, \cite{marsden-ratiu-weinstein}, \cite{montgomery}.

\vskip 0.2cm 
\noindent
 {\bf Acknowledgements.} The authors would like to thank the associated editor,  the referees, 
Prof. K. H. Neeb, O. R. Abib, M. N. Boyom and E. Okassa for 
 valuable remarks and suggestions that help improve the present paper. 
 The second author's Ph.D. thesis was funded  by ICTP-Trieste (Italy) through the ICAC-3 Programme, a part of 
 this work was started during his visit to Universit\'e Montpellier II  
 (France) funded by the SARIMA project, whereas, the last version was written during his visit to IMSP (Benin) funded by the DAAD.  To all these institutions, he wishes to express his gratitude and thanks.


\label{lastpage-01}
\end{document}